%% file: main.tex
\pdfoutput=1
\documentclass{article}
\usepackage[a4paper, margin=1.5in]{geometry}  

\usepackage[T1]{fontenc}      
\usepackage[utf8]{inputenc}   
\usepackage{microtype}        

\usepackage{amsmath}
\usepackage{amsthm}
\usepackage{amssymb}
\usepackage{commath}
\usepackage{mathtools}
\usepackage{thmtools}
\usepackage{mathrsfs}

\usepackage{hyperref}
\hypersetup{
  colorlinks=true,
  citecolor=blue,
  linkcolor=blue,
  }
\usepackage{cleveref}
\usepackage[numbers]{natbib}
\usepackage[dvipsnames]{xcolor} 

\usepackage{sectsty}
\allsectionsfont{\sffamily}

\usepackage{titlesec}
\titleformat{\subsubsection}[runin]
  {\normalfont\itshape}
  {\thesubsubsection}
  {1ex}
  {\addperiod}
\titleformat{\paragraph}[runin]
  {\normalfont\itshape}
  {\theparagraph}
  {1ex}
  {\addperiod}
\newcommand{\addperiod}[1]{#1.}

\let\oldabstractname\abstractname
\renewcommand*\abstractname{\sffamily\oldabstractname}

\numberwithin{equation}{section}
\declaretheorem[Refname={Theorem,Theorems}]{theorem}
\numberwithin{theorem}{section} 

\declaretheorem[style=definition,numberlike=theorem,Refname={Remark,Remarks}]{remark}
\declaretheorem[numberlike=theorem,Refname={Lemma,Lemmas}]{lemma}
\declaretheorem[name=Corollary,numberlike=theorem,Refname={Corollary,Corollaries}]{corollary}
\declaretheorem[name=Proposition,numberlike=theorem,Refname={Proposition,Propositions}]{proposition}

\DeclarePairedDelimiterX\Set[2]{\lbrace}{\rbrace}%
{ #1 \,:\, #2 }                                         
\DeclarePairedDelimiterX\inprod[2]{\langle}{\rangle}%
{ #1 , #2 }                                             
\DeclarePairedDelimiter{\floor}{\lfloor}{\rfloor}         

\renewcommand{\b}[1]{\mathbf{#1}}

\newcommand{\R}{\mathbb{R}} 
\newcommand{\N}{\mathbb{N}} 
\newcommand{\Z}{\mathbb{Z}} 
\newcommand{\C}{\mathbb{C}} 

\newcommand{\ri}{\mathrm{i}} 

\title{\sffamily Approximation in Hilbert spaces of the Gaussian and related analytic kernels}
\author{\sffamily Toni Karvonen$^{1,2}$ --- Yuya Suzuki$^3$}
\date{
  {     
    \normalsize
    \sffamily
    $^1$School of Engineering Sciences, Lappeenranta--Lahti University of Technology LUT, Finland\\
    \vspace{0.1cm}
    $^2$Department of Mathematics and Statistics, University of Helsinki, Finland \\
    \vspace{0.1cm}
    $^3$Department of Mathematics and Systems Analysis, Aalto University, Finland \\
  }
  \vspace{0.4cm}
  \sffamily\today
}

\begin{document}

\maketitle

\begin{abstract}
  \noindent
  We consider linear approximation based on function evaluations in reproducing kernel Hilbert spaces of certain analytic weighted power series kernels and stationary kernels on the interval $[-1,1]$.
  Both classes contain the popular Gaussian kernel $K(x, y) = \exp(-\tfrac{1}{2}\varepsilon^2(x-y)^2)$.
  For weighted power series kernels we derive almost matching upper and lower bounds on the worst-case error.
  When applied to the Gaussian kernel, our results state that, up to a sub-exponential factor, the $n$th minimal error decays as $(\varepsilon/2)^n (n!)^{-1/2}$.
  The proofs are based on weighted polynomial interpolation and classical polynomial coefficient estimates that we use to bound the Hilbert space norm of a weighted polynomial fooling function.
\end{abstract}

\section{Introduction} \label{sec:intro}

\input{sec1.tex}

\section{Prelude: Korobov spaces}

\input{sec2.tex}

\section{Weighted power series kernels} \label{sec:power-series}

\input{sec3.tex}

\section{Stationary kernels} \label{sec:stationary}

\input{sec4.tex}

\section*{Acknowledgements}
The authors were supported by the Research Council of Finland decisions 338567, 348503, 359181 and 359183.


\end{document}

%% file: sec1.tex
Let $K \colon [-1,1] \times [-1,1] \to \R$ be a positive-semidefinite kernel on the interval $[-1, 1]$ and $H(K)$ its reproducing kernel Hilbert space equipped with inner product $\langle \cdot, \cdot \rangle_K$ and norm $\lVert \cdot \rVert_K$.
This space consists of certain functions defined on $[-1,1]$.
Properties of the kernel determine which functions are contained in $H(K)$.
For example, if $K$ is $m$ times differentiable in both arguments, all functions in $H(K)$ are $m$ times differentiable; if $K$ is analytic, $H(K)$ consists of analytic functions~\cite[Sec.\@~2.1]{SaitohSawano2016}.
In this paper we study approximation in $H(K)$ when the reproducing kernel is analytic and obtain reasonably tight upper and lower bounds on three types of worst-case errors.

Let $\lVert \cdot \rVert_p$ denote the $L^p = L^p([-1,1])$-norm and let $x_1, \ldots, x_n \in [-1, 1]$ be any points.
We assume that we have access to function evaluations at these points, a type of data known as standard information.
Firstly, we study the minimal worst-case error
\begin{equation} \label{eq:wce1}
  e_p^{\min}(x_1, \ldots, x_n) = \inf_{\psi_1, \ldots, \psi_n \in L^p} \sup_{ 0 \neq f \in H(K) } \frac{ \lVert f - \sum_{k=1}^n f(x_k) \psi_k \rVert_p }{\lVert f \rVert_{K}}
\end{equation}
of linear approximation given standard information at the points.
Secondly, we consider the $n$th minimal error 
\begin{equation} \label{eq:wce2}
        e_p^{\min}(n) = \inf_{x_1, \ldots, x_n} e_p(x_1, \ldots, x_n) = \inf_{\substack{x_1, \ldots, x_n \\ \psi_1, \ldots, \psi_n \in L^p}} \sup_{ 0 \neq f \in H(K) } \frac{ \lVert f - \sum_{k=1}^n f(x_k) \psi_k \rVert_p }{\lVert f \rVert_{K}} ,
\end{equation}
which is the smallest approximation error that can be attained with $n$ function evaluations.
Finally, we also study the pointwise worst-case error
\begin{equation} \label{eq:wce-x}
        e^{\min}(x \mid x_1, \ldots, x_n) = \inf_{u_1, \ldots, u_n \in \R} \sup_{ 0 \neq f \in H(K) } \frac{ \lvert f(x) - \sum_{k=1}^n f(x_k) u_k \rvert }{\lVert f \rVert_{K}} ,
\end{equation}
which only tells us about the error at a single fixed point $x \in [-1, 1]$.
The pointwise worst-case error equals the power function in the radial basis function interpolation~\cite{Schaback1993, Wendland2005} and the posterior or kriging variance in Gaussian process modelling~\cite{RasmussenWilliams2006, Scheuerer2013, Stein1999}; see also \Cref{sec:yarotsky}.
The original motivation for this work was to obtain bounds on the pointwise worst-case error that could be used to derive asymptotics for covariance kernel parameter estimators in Gaussian process modelling in a manner similar to~\cite{Karvonen2022-smoothness}.

We are in particular interested in the Gaussian kernel $K(x, y) = \exp(-\tfrac{1}{2} \varepsilon^2 (x - y)^2)$, where $\varepsilon > 0$ is a shape parameter, and its Hilbert space.
Although strong arguments have been made against its use in applications~\cite[p.\@~30]{Stein1999}, the Gaussian kernel remains popular, especially in machine learning~\cite[p.\@~83]{RasmussenWilliams2006}.
We study two classes of kernels that both contain the Gaussian kernel: \emph{weighted power series kernels} given by
\begin{equation} \label{eq:power-series-intro}
        K(x, y) = \varphi(x) \varphi(y) \sum_{k=0}^\infty \alpha_k^{-1} x^k y^k
\end{equation}
for a suitable function $\varphi \colon [-1, 1] \to \R$ and positive coefficients $\alpha_k$, and \emph{stationary kernels} defined by a positive-semidefinite function $\Phi \colon \R \to \R$ via
\begin{equation} \label{eq:stationary-intro}
        K(x, y) = \Phi(x - y) .
\end{equation}
The Gaussian kernel is obtained by setting $\alpha_k = \varepsilon^{-2k} k!$ and $\varphi(x) = e^{-\frac{1}{2}\varepsilon^2 x^2}$ in~\eqref{eq:power-series-intro} or $\Phi(z) = \exp(-\tfrac{1}{2} \varepsilon^2 z^2)$ in~\eqref{eq:stationary-intro}.
How fast the coefficients $\alpha_k$ blow up or how fast the Fourier transform of $\Phi$ tends to zero at infinity determines how smooth these kernels and their Hilbert spaces are.
It would be natural to consider these kernels and approximation in their Hilbert spaces on the entire real line, rather than on a bounded interval as we do. However, the techniques that we use exploit the properties of polynomials on bounded intervals and do not extend easily to other cases.

In \Cref{sec:power-series} we prove that the worst-case errors in~\eqref{eq:wce1}--\eqref{eq:wce-x} decay as $\alpha_n^{-1/2}$, up to an exponential factor, when $K$ is a weighted power series kernel in~\eqref{eq:power-series-intro} whose coefficients blow up factorially or faster.
In \Cref{sec:stationary} we prove analogous lower bounds when $K$ is a stationary kernel in~\eqref{eq:stationary-intro} defined by $\Phi$ whose Fourier transform tends to zero sufficiently fast.
Our main results are \Cref{thm:main-3,thm:main-4}.
By applying \Cref{thm:main-3} to the Gaussian kernel we obtain the following result.

\begin{theorem} \label{thm:intro-gaussian}
  Let $K(x, y) = \exp(-\tfrac{1}{2} \varepsilon^2 (x - y)^2)$ be the Gaussian kernel, $p \in \{2, \infty\}$, and $c_L$ and $m_L$ the functions defined in~\eqref{eq:cL-constant} and~\eqref{eq:mL-constant}.
  If $x_1, \ldots, x_n \in [-1, 1]$ are distinct, then
  \begin{align}
        c_{1,p} \bigg( \frac{\varepsilon}{2} \bigg)^n (n!)^{-1/2} &\leq e_p^{\min}(x_1, \ldots, x_n) \leq c_{2,p} \, n^{-1/8 - 1/p} e^{\varepsilon\sqrt{n}} \, (2 \varepsilon)^{n} (n!)^{-1/2} , \label{eq:intro-31} \\
        c_{1,p} \bigg( \frac{\varepsilon}{2} \bigg)^n (n!)^{-1/2} &\leq e_p^{\min}(n) \leq 2 c_{2,p} \, n^{-1/8} e^{\varepsilon \sqrt{n}} \bigg( \frac{\varepsilon}{2} \bigg)^n (n!)^{-1/2} , \label{eq:intro-32} \\
        c_{1,\infty} \bigg( \frac{\varepsilon}{4} \bigg)^n (n!)^{-1/2} &\leq \frac{e^{\min}(x \mid x_1, \ldots, x_n)}{\lvert \prod_{k=1}^n (x - x_k) \rvert} \leq c_{2,p} \, n^{-1/8} e^{\varepsilon \sqrt{n}} \varepsilon^n (n!)^{-1/2} \label{eq:intro-33} ,
  \end{align}
where the lower bounds hold when $n \geq 1$ and the upper bounds when $n \geq m_L(\varepsilon^2)$, the constants $c_{1, \infty}$ and $c_{1,2}$ are obtained by setting \smash{$\varphi_{\min} = e^{-\frac{1}{2} \varepsilon^2}$} and $\lambda = \varepsilon^2$ in~\eqref{eq:constants-main-intro}, and $c_{2,\infty} = \sqrt{\smash[b]{c_L(\varepsilon^2)}}$ and $c_{2,2} = \sqrt{\smash[b]{2 c_L(\varepsilon^2)}}$.
\end{theorem}

The constant coefficients in the bounds could be slightly improved but we have opted for simplicity over optimality.
The most remarkable feature of these bounds is that~\eqref{eq:intro-32} shows that the rate of decay of the $n$th minimal error is controlled by the shape parameter $\varepsilon$.
That is,
\begin{equation} \label{eq:scale-dependency}
        \frac{e_{p,\varepsilon}^{\min}(n)}{e_{p,\varepsilon'}^{\min}(n)} \to 0 \: \text{ if } \: \varepsilon < \varepsilon' \quad \text{ and } \quad \frac{e_{p,\varepsilon}^{\min}(n)}{e_{p,\varepsilon'}^{\min}(n)} \to \infty \: \text{ if } \: \varepsilon > \varepsilon',
\end{equation}
where the additional subscript indicates dependency on the shape parameter of the Gaussian kernel. 
To understand how estimators of the parameter $\varepsilon$ behave in Gaussian process regression it is crucial to have results such as~\eqref{eq:scale-dependency}.
For example, the maximum likelihood estimate of $\varepsilon$ given data (i.e., function evaluations) at points $x_1, \ldots, x_n$ is
\begin{equation*}
        \hat{\varepsilon}_n = \operatorname*{arg\,min}_{\varepsilon > 0} \big\{ \mathrm{DF}(\varepsilon \mid x_1, \ldots, x_n) + \mathrm{MC}(\varepsilon \mid x_1, \ldots, x_n ) \big\} ,
\end{equation*}
where the former term is a data-fit term that depends on the function evaluations and the latter a model complexity term that depends only on $\varepsilon$ and the points.
The model complexity (we do not discuss the data-fit here) can be written as
\begin{equation*}
        \mathrm{MC}(\varepsilon \mid x_1, \ldots, x_n ) = 2 \sum_{i=1}^n \log e_\varepsilon^{\min}(x_i \mid x_1, \ldots, x_{i-1}),
\end{equation*}
where the first summand is the logarithm of the pointwise initial error at $x_1$ and the $\varepsilon$-dependency is made explicit~\cite[Sec.\@~2.3]{Karvonen2022-smoothness}.
Precise control of the pointwise worst-case error with respect to $\varepsilon$ is therefore essential in order to figure out how $\hat{\varepsilon}_n$ behaves as $n \to \infty$.
To the best of our knowledge, no result resembling~\eqref{eq:scale-dependency} has been proved for the Gaussian kernel or its relatives that we study.
In a paper that has regrettably gone unnoticed outside of Gaussian process literature, Yarotsky~\cite{Yarotsky2013} has proved a lower bound on the pointwise worst-case error that is slightly better than~\eqref{eq:intro-33}.
However, his upper bound is looser than ours.
We review Yarotsky's result in more detail in \Cref{sec:yarotsky}.

Upper bounds on the Gaussian and other analytic kernels abound in the scattered data approximation literature.
These are usually expressed in terms of the power function, which is nothing but the pointwise worst-case error $e^{\min}(x \mid x_1, \ldots, x_n)$.
Typical results for the Gaussian kernel state that, for any $1 \leq p \leq \infty$ and some constants $c$ and $\gamma$,
\begin{equation} \label{eq:power-function-bound}
  \norm[0]{ e^{\min}(\cdot \mid x_1, \ldots x_n)}_p \leq c \, e^{- \gamma n \log n} = c \, n^{- \gamma n}
\end{equation}
if the points $x_1, \ldots, x_n$ are sufficiently uniform on $[-1, 1]$; see~\cite[Thm.\@~11.22]{Wendland2005} and~\cite[Thm.\@~6.1]{RiegerZwicknagl2010}.
Bounds for various weighted power series kernels with $\varphi \equiv 1$ may be found in \cite{Zwicknagl2009, ZwicknaglSchaback2013}.
Because~\eqref{eq:power-function-bound} is a special case of a bound that holds for the Gaussian kernel defined in any dimension and is proved by means of local polynomial reproduction~\cite[Ch.\@~3]{Wendland2005}, the constant $\gamma$ is rather complicated.
For equispaced points on $[-1, 1]$ one may compute that $\gamma \leq 1/48$~\cite[Sec.\@~4.2]{KarvonenTanaka2021}, which gives an extremely sub-optimal rate since Stirling's formula shows that all bounds in \Cref{thm:intro-gaussian} are, up to exponential factors, of order $n^{-n/2}$.
Although apparently absent from the literature with the exception of~\cite{KarvonenTanaka2021}, it is difficult to believe that upper bounds of order $n^{-n/2}$ are not part of the folklore of the field because, as we shall see in \Cref{sec:applications}, proving them in dimension one via polynomial interpolation is a simple exercise.
Connections between polynomial interpolation and interpolation with translates of the Gaussian kernel have been explored in the fascinating papers~\cite{Platte2011, PlatteDriscoll2005}.
Moreover, recent years have seen a host of results on approximation and integration in the Hilbert space of the Gaussian kernel when error is measured in the Gaussian-weighted $L^2$-norm over $\R^d$~\cite{ChenWang2019,FasshauerHickernell2010,FasshauerHickernell2012, Karvonen2021, Karvonen2019, KhartovLimar2022, KhartovLimar2023, KuoSloanWozniakowski2017,KuoWozniakowski2012, SloanWozniakowski2018}.
Results on covering numbers and regression rates for the Gaussian kernel may be found in~\cite{EbertsSteinwart2013, Kuhn2011, SteinwartFischer2021, VaartZanten2009, Zhou2002, Zhou2003}.
Closely related Hermite spaces have been studied in~\cite{Gnewuch2024, Irrgeher2015, Irrgeher2016}.
See also~\cite{AnderssonBojanov1984} for optimal rates of integration in the Hardy space of analytic functions that is obtained by setting $\varphi \equiv 1$ and $\alpha_k = r^{2k}$ for $r \neq 0$ in~\eqref{eq:power-series-intro}.
Chapter~VIII in~\cite{Pinkus1985} reviews various results on $n$-widths of spaces of analytic functions.

Our upper bounds are based on weighted polynomial interpolation, bounds on Laguerre polynomials, and the standard error estimate for polynomial interpolation.
We use the standard fooling function technique to obtain the lower bounds (see~\cite{KriegVybiral2023} for some commentary on this technique).
That is, for a set of $n$ points we construct a function $f_n \in H(K)$ that vanishes at these points, so that the approximations in~\eqref{eq:wce1}--\eqref{eq:wce-x} are identically zero and the worst-case errors are bounded from below by $\lVert f_n \rVert_p / \lVert f_n \rVert_K$ or $\lvert f_n(x) \rvert / \lVert f_n \rVert_K$.
The difficulty is then in estimating the $H(K)$-norm of $f_n$ from above.
If $H(K)$ were a Sobolev space, this would be relatively easy as $f_n$ could be constructed out of bump functions whose Sobolev norms are straightforward to estimate.
Because the spaces we consider consist of analytic functions, bump functions are out of question.
Instead, we select $f_n$ as a weighted polynomial interpolant, a more or less trivial construction.
However, to estimate the $H(K)$-norm of such an $f_n$ is not straightfoward.
To do this we use the classical Markov inequality on the maxima of the derivatives and coefficients of a polynomial from 1892 and its $L^2$-version by Labelle~\cite{Labelle1969}.
To the best of our knowledge these results have not been used to bound worst-case errors before.
We begin by illustrating this technique in the particularly tractable setting of Korobov spaces.

%% file: sec2.tex
Suppose that the reproducing kernel on $[-1,1]$ is given by
\begin{equation*}
  K(x, y) = \sum_{k \in \Z} \lambda_k^{-1} \;\frac{e^{\pi \ri k x}}{\sqrt{2}} \; \overline{\frac{e^{\pi \ri k y}}{\sqrt{2}}} = \sum_{k \in \Z} \lambda_k^{-1} \;\frac{e^{\pi \ri k x}}{\sqrt{2}} \; \frac{e^{-\pi \ri k y}}{\sqrt{2}}
\end{equation*}
for positive $\lambda_k$.
Here $\overline{z}$ stands for the complex conjugate of $z \in \C$.
The reproducing kernel Hilbert space $H(K)$ is then~\cite[Sec.\@~2.1]{Paulsen2016}
\begin{equation} \label{eq:korobov-space}
  H(K) = \Set[\bigg]{ f(x) = \sum_{k \in \Z} c_k \frac{e^{\pi \ri k x}}{\sqrt{2}} }{ \lVert f \rVert_K^2 = \sum_{k \in \Z} \lambda_k \lvert c_k \rvert^2  < \infty } .
\end{equation}
Because $\langle 2^{-1/2} e^{\pi \ri \ell x}, 2^{-1/2} e^{\pi \ri k x}\rangle_{2}= \tfrac{1}{2} \int_{-1}^1 e^{\pi \ri (\ell-k) x} \dif x = \delta_{\ell k}$, where $\delta_{\ell k}$ is Kronecker's delta, the function $2^{-1/2} e^{\pi \ri k x}$ is the $k$th eigenfunction and \smash{$\lambda_k^{-1}$} the corresponding eigenvalue of the integral operator $T_K \colon H(K) \to H(K)$ given by
\begin{equation} \label{eq:TK}
        (T_K f)(x) = \int_{-1}^1 K(x, y) f(y) \dif y .
\end{equation}
If $\lambda_k$ blows up as $\lvert k \rvert^{2s}$, the Hilbert space is a Sobolev space of periodic functions on $[-1, 1]$ with smoothness $s$, also known as a Korobov space.
If $\lambda_k$ grow exponentially fast, $H(K)$ consists of certain periodic analytic functions.
The proof of the following proposition is analogous to the proofs for lower bounds in \Cref{sec:power-series,sec:stationary}.

\begin{proposition} \label{prop:korobov}
If $n$ is even and $x_1, \ldots, x_n \in [-1, 1]$ are any points, then 
  \begin{equation} \label{eq:korobov-lower-bound}
    e_2^{\min}(x_1, \ldots, x_n) \geq  \bigg( \sum_{k=-n/2}^{n/2} \lambda_k \bigg)^{-1/2}.
  \end{equation}
\end{proposition}
\begin{proof}
  Let $a \in (-1, 1)$ be not equal to any of $x_1, \ldots, x_n$.
  It is clear that the function
  \begin{equation*}
    f_n(x) = \prod_{k=1}^n \frac{\sin(\tfrac{1}{2}\pi(x-x_k))}{\sin(\tfrac{1}{2}\pi(a - x_k))}
  \end{equation*}
  vanishes at $x_1, \ldots, x_n$ and equals one at $a$.
  This function is a trigonometric polynomial of order $n/2$ and has the form
  \begin{equation}
    f_n(x) = \sum_{k =  -n/2 }^{ n/2 } c_k \frac{e^{\pi \ri k x}}{\sqrt{2}}
  \end{equation}
  for some coefficients $c_k \in \C$; see~\cite[Sec.\@~1 in Ch.\@~X]{Zygmund2003}.  
Parseval's identity tells us that
\begin{equation} \label{eq:korobov-coef-bound}
 \lVert f_n \rVert_2 =   \bigg(\sum_{\ell =  -n/2 }^{ n/2 } \lvert c_\ell \rvert^2\bigg)^{1/2} \ge \lvert c_k \rvert
\end{equation}
for $-n/2 \leq k \leq n/2$.
Consequently, the Hilbert space characterisation in~\eqref{eq:korobov-space} gives
\begin{equation} \label{eq:korobov-norm-estimate}
        \lVert f_n \rVert^2 _K= \sum_{k=-n/2}^{n/2} \lambda_k \lvert c_k \rvert^2 \leq  \lVert f_n \rVert_2^2 \sum_{k=-{n/2}}^{n/2} \lambda_k .
\end{equation}
Because $f_n$ vanishes at $x_1, \ldots, x_n$, the function $\sum_{k=1}^n f_n(x_k) \psi_k$ is identically zero for any $\psi_k \in L^2$.
Therefore~\eqref{eq:wce1} and~\eqref{eq:korobov-norm-estimate} yield
\begin{equation*}
    e_2^{\min}(x_1, \ldots, x_n) \geq \frac{ \lVert f_n \rVert_2 }{\lVert f_n \rVert_{K}} \geq \frac{ \lVert f_n \rVert_2 }{  \lVert f_n \rVert_2 (\sum_{k=-n/2}^{n/2} \lambda_k)^{1/2} } =\bigg( \sum_{k=-n/2}^{n/2} \lambda_k \bigg)^{-1/2}. \qedhere
\end{equation*}  
\end{proof}

The purpose of \Cref{prop:korobov} is only to illustrate our proof technique.
It is well known that square root of the $n$th largest eigenvalue of the integral operator~\eqref{eq:TK} equals the Kolmogorov $n$-width in $L^2$, which is the smallest possible worst-case error in $L^2$ given $n$ pieces of linear information~\cite[Cor.\@~4.12]{NovakWozniakowski2008}.
If $\lambda_k = \lambda_{-k}$ and $\lambda_{\lvert k \rvert}$ is non-decreasing in $\lvert k \rvert$, this implies the lower bound
\begin{equation*}
  e_2^{\min}(x_1, \ldots, x_n) \geq e_2^{\min}(n) \geq \lambda_{n/2}^{-1/2} ,
\end{equation*}
which is tighter than~\eqref{eq:korobov-lower-bound}.
In \Cref{sec:power-series,sec:stationary} no such trivial lower bound is available because we do not have access to the eigenvalues of $T_K$ and the polynomial basis functions we work with are not $L^2$-orthonormal.
Moreover, polynomial coefficient bounds of the form~\eqref{eq:korobov-coef-bound} involve additional factors on the right-hand side, which complicates estimation of the sum in~\eqref{eq:korobov-lower-bound}.

It is not surprising that the above technique yields sub-optimal lower bounds: the coefficient estimate~\eqref{eq:korobov-coef-bound} holds for \emph{all} trigonometric polynomials, not just those that vanish at $x_1, \ldots, x_n$.
That is, not all properties of the fooling function are exploited in the proof and it would be unrealistic to expect a tight lower bound.
However, what is somewhat surprising is that such a crude technique nevertheless suffices for the almost tight lower bound in~\eqref{eq:intro-31} for the $n$th minimal error.

%% file: sec3.tex
Let $\varphi \colon [-1, 1] \to \R$ be a non-vanishing Lebesgue measurable function and $(\alpha_k)_{k = 0}^\infty$ a positive sequence such that $\sum_{k=0}^\infty \alpha_k^{-1} < \infty$.
It is easy to see that
\begin{equation} \label{eq:kernel}
  K(x, y) = \varphi(x) \varphi(y) \sum_{k=0}^\infty \alpha_k^{-1} x^k y^k
\end{equation}
is a well-defined strictly positive-definite kernel on $[-1, 1] \times [-1, 1]$.
When $\varphi \equiv 1$, these kernels are called power series kernels~\cite{Zwicknagl2009, ZwicknaglSchaback2013}.
Accordingly, we call any kernel of the form~\eqref{eq:kernel} a weighted power series kernel.
From~\eqref{eq:kernel} it follows that~\cite[Sec.\@~2.1]{Paulsen2016} the Hilbert space in which $K$ is reproducing is 
\begin{equation} \label{eq:rkhs-norm}
        H(K) = \Set[\bigg]{ f(x) = \varphi(x) \sum_{k=0}^\infty c_k x^k }{ \lVert f \rVert_K^2 = \sum_{k=0}^\infty \alpha_k c_k^2 < \infty } .
\end{equation}
That is, the rate of growth of $\alpha_k$ determines how large $H(K)$ is.
We shall make two assumptions.
Firstly, we assume that there are positive constants $\varphi_\textup{min}$ and $\varphi_\textup{max}$ such that 
\begin{equation} \label{eq:phi-assumption}
  0 < \varphi_\textup{min} \leq \abs[0]{\varphi(x)} \leq \varphi_\textup{max} < \infty 
\end{equation}
for all $x \in [-1, 1]$.
This is a very mild assumption.
Our second assumption is that there exists a positive constant $\lambda$ such that
\begin{equation} \label{eq:alpha-assumption}
  \frac{\alpha_{k}}{\alpha_n} \leq \lambda^{n-k} \frac{k!}{n!} 
\end{equation}
for all $n \geq 0$ and $0 \leq k \leq n$.
Implying that $\alpha_k \geq \alpha_0 \lambda^{-k} k!$ for every $k \geq 0$, this assumption is far from necessary to ensure that $H(K)$ consists of analytic functions.
Why we employ~\eqref{eq:alpha-assumption} is related to deficiencies in the technique we use to obtain lower bounds and is discussed in more detail in \Cref{rmk:alpha-assumption-unif,rmk:alpha-assumption-l2}.
Note that we have not imposed any smoothness assumptions on the function $\varphi$, the qualitative properties of which play no role in our results.
While a non-smooth $\varphi$ yields a Hilbert space of non-smooth functions, these functions are ``morally'' smooth, being products of a smooth function and a known non-smooth function.
It is necessary to include $\varphi$ if our results are to apply to the Gaussian kernel.

We are aware of four weighted power series kernels that have closed form expressions, satisfy~\eqref{eq:phi-assumption} and~\eqref{eq:alpha-assumption}, and have appeared in the literature.
Let $\varepsilon > 0$ be a shape parameter.
The Gaussian kernel
\begin{equation} \label{eq:gaussian-kernel-intro}
    K(x, y) = \exp\bigg(\! -\frac{1}{2} \varepsilon^2 (x - y)^2 \bigg)
\end{equation}
is obtained by selecting $\alpha_k = \varepsilon^{-2k} k!$ and $\varphi(x) = e^{-\frac{1}{2}\varepsilon^2 x^2}$.
The exponential kernel $K(x, y) = e^{\varepsilon xy}$ is obtained by selecting $\alpha_k = \varepsilon^{-k} k!$ and $\varphi \equiv 1$.
Let $\tau > 0$.
The analytic Hermite kernel (or Mehler kernel)
\begin{equation*}
    K(x, y) = \exp\bigg( \! - \frac{1}{2} \varepsilon^2(x^2 + y^2)  + \tau^2 \varepsilon^2 xy \bigg)
\end{equation*}
is obtained by selecting $\alpha_k = (\tau \varepsilon)^{-2k} k!$ and $\varphi(x) = e^{-\frac{1}{2} \varepsilon^2 x^2}$.
If $\tau = 1$, this kernel reduces to the Gaussian kernel.
But taking any $r \in (0, 1)$ and setting $\varepsilon^2 = r^2/(1 - r^2)$ and $\tau^2 = 1/r$ yields, via Mehler's formula,
\begin{equation*}
    K(x, y) = \exp\bigg( \! - \frac{r^2(x^2 + y^2) - 2rxy}{2(1 - r^2)} \bigg) = \sqrt{1 - r^2} \, \sum_{k=0}^\infty \frac{r^k}{k!} H_k(x) H_k(y),
\end{equation*}
where $H_k$ are the probabilist's Hermite polynomials; see~\cite[Sec.\@~3.1.2]{IrrgeherLeobacher2015} and~\cite[Sec.\@~5.2]{Karvonen2022-sample}.
Finally, let $\mathcal{I}_0$ be the modified Bessel function of the first kind of order zero.
The Bessel kernel $K(x, y) = \mathcal{I}_0( 2 \varepsilon \sqrt{\smash[b]{x y}} \, )$ is obtained by selecting $\alpha_k = \varepsilon^{-k} k!^2$ and $\varphi \equiv 1$~\cite[p.\@~64]{Zwicknagl2009}.

The bulk of this section is taken up by the proof of the following theorem that provides upper and lower bounds on the worst-case errors in~\eqref{eq:wce1}--\eqref{eq:wce-x}.

\begin{theorem} \label{thm:main-3}
    Let $K$ be a weighted power series kernel in~\eqref{eq:kernel} and $p \in \{2, \infty\}$.
    Suppose that $\varphi$ and $(\alpha_k)_{k=0}^\infty$ satisfy~\eqref{eq:phi-assumption} and~\eqref{eq:alpha-assumption}.
    Let $c_L$ and $m_L$ be the functions defined in~\eqref{eq:cL-constant} and~\eqref{eq:mL-constant}.
    If $x_1, \ldots, x_n \in [-1, 1]$ are distinct, then
    \begin{align}
          c_{1,p} \, 2^{-n} \alpha_n^{-1/2} &\leq e_p^{\min}(x_1, \ldots, x_n) \leq c_{2, p} \, \varphi_{\max} \, n^{-1/8 - 1/p} e^{\sqrt{\lambda n}} \, 2^{n} \alpha_n^{-1/2} , \label{eq:main-31} \\
          c_{1,p} \, 2^{-n} \alpha_n^{-1/2} &\leq e_p^{\min}(n) \leq 2 c_{2, p} \, \varphi_{\max} \, n^{-1/8} e^{\sqrt{\lambda n}} \,  2^{-n}  \alpha_n^{-1/2} , \label{eq:main-32} \\
          c_{1,\infty} \, 2^{-2n} \alpha_n^{-1/2} &\leq \frac{e^{\min}(x \mid x_1, \ldots, x_n)}{\lvert \prod_{k=1}^n (x - x_k) \rvert} \leq c_{2, \infty} \, \varphi_{\max} \, n^{-1/8} e^{\sqrt{\lambda n}} \alpha_n^{-1/2} \label{eq:main-33} ,
    \end{align}
  where the lower bounds hold when $n \geq 1$ and the upper bounds when $n \geq m_L(\lambda)$,
    \begin{equation} \label{eq:constants-main-intro}
      c_{1,\infty} = \sqrt{ \frac{2}{1 + \lambda}} \, e^{-\frac{1}{32}\lambda^2} \varphi_\textup{min} \quad \text{ and } \quad c_{1, 2} = \frac{2\sqrt{2\pi}}{3e^3 \sqrt{\smash[b]{3(4 + \lambda)}}} \, e^{-\frac{1}{32}\lambda^2} \varphi_\textup{min},
    \end{equation}
    and $c_{2,\infty} = \sqrt{\smash[b]{c_L(\lambda)}}$ and $c_{2,2} = \sqrt{\smash[b]{2 c_L(\lambda)}}$.
\end{theorem}

The lower bounds in \Cref{thm:main-3} are proved in \Cref{sec:lower-Linf,sec:lower-L2} and the upper bounds in \Cref{sec:upper-bounds}.
\Cref{sec:applications} applies the above theorem to the weighted power series kernels mentioned above. 
Bounds on worst-case errors given derivative information are derived in \Cref{sec:derivative-information}.
Formally this corresponds to setting $x_1 = \cdots = x_n = a \in (-1, 1)$ in~\eqref{eq:main-31}.
Some remarks on extensions to higher dimensions can be found in \Cref{sec:higher-dimensions}.

\subsection{Lower bounds for approximation in $L^\infty([-1, 1])$} \label{sec:lower-Linf}

In this section we prove lower bounds on $e_\infty^{\min}(x_1, \ldots, x_n)$, $e_\infty^{\min}(n)$, and $e^{\min}(x \mid x_1, \ldots, x_n)$.
The Chebyshev polynomial of the first kind of degree $n$ is 
\begin{equation*}
  T_n(x) = \sum_{k=0}^n t_{n,k} x^k,
\end{equation*}
where the non-zero coefficients are
\begin{equation} \label{eq:chebyshev-coefs}
  t_{n, n - 2k} = (-1)^k \frac{2^{n-2k-1} n (n-k-1)!}{k! (n - 2k)!} \quad \text{ for } \quad 0 \leq k \leq \floor[\bigg]{\frac{n}{2}}.
\end{equation}
Note that the leading coefficient, $t_{n,n}$, equals $2^{n-1}$.
The roots of $T_n$ are the Chebyshev nodes
\begin{equation} \label{eq:chebyshev-nodes}
  x_{n,k} = \cos\bigg( \frac{\pi(k - 1/2)}{n} \bigg) \in (-1, 1) \quad \text{ for } \quad 1 \leq k \leq n.
\end{equation}
The extrema, which are either $1$ or $-1$, of $T_n$ on $[-1, 1]$ are located at
\begin{equation} \label{eq:chebyshev-extrema}
  \tilde{x}_{n,k} = \cos(\pi k/n) \quad \text{ for } \quad 0 \leq k \leq n.
\end{equation}
That is, $T_n(x) \in [-1, 1]$ for all $x \in [-1, 1]$.
For our purposes the importance of Chebyshev polynomials lies in a classical result by V.\ A.\ Markov from 1892 (for a brief history, see~\cite[p.\@~679]{RahmanSchmeisser2002}) and its generalisations~\cite[Sec.\@~16.3]{RahmanSchmeisser2002} that $T_n$ has the maximal coefficients among polynomials of degree~$n$ that take values in $[-1, 1]$.
Specifically~\cite[Thm.\@~16.3.3]{RahmanSchmeisser2002}, if $P_n(x) = \sum_{k=0}^n a_k x^k$ is a polynomial of degree $n$ with real coefficients $a_k$ and $\abs[0]{P_n(\tilde{x}_{n,k})} \leq 1$ for every $0 \leq k \leq n$, then
\begin{equation} \label{eq:general-markov-inequality}
  \abs[0]{a_{n-2k}} + \abs[0]{a_{n-2k-1}} \leq \abs[0]{t_{n,n-2k}} \quad \text{ for every } \quad 0 \leq k \leq \floor[\bigg]{\frac{n-1}{2}}.
\end{equation}
The lemmas below use this result to estimate the $H(K)$-norm of the weighted polynomial $f_n(x) = \varphi(x) P_n(x)$.

\begin{lemma} \label{lemma:chebyshev-c-lemma}
  Suppose that assumption~\eqref{eq:alpha-assumption} holds.
  Then for every $n \geq 1$ we have
  \begin{equation} \label{eq:chebyshev-norm-lemma}
    \tau_1 2^{2n} \alpha_n \leq \sum_{k=0}^{\floor{n/2}} \alpha_{n - 2k} t_{n, n - 2k}^2 \leq \tau_2 2^{2n} \alpha_n, 
  \end{equation}
  where
  \begin{equation} \label{eq:chebyshev-norm-lemma-constants}
    \tau_1 = \frac{1}{4} \quad \text{ and } \quad \tau_2 = \frac{1}{2} e^{\frac{1}{16} \lambda^2 }.
  \end{equation}
\end{lemma}
\begin{proof}
  Equation~\eqref{eq:chebyshev-coefs} gives
  \begin{equation*}
    \begin{split}
      \sum_{k=0}^{\floor{n/2}} \alpha_{n - 2k} t_{n, n - 2k}^2 &= \alpha_n t_{n,n}^2 \sum_{k=0}^{\floor{n/2}} \frac{\alpha_{n-2k}}{\alpha_n} \bigg( \frac{t_{n, n - 2k}}{t_{n,n}} \bigg)^2 \\
      &= 2^{2n-2} \alpha_n \sum_{k=0}^{\floor{n/2}} \frac{\alpha_{n-2k}}{\alpha_n} \bigg(\frac{n (n-k-1)!}{(n-2k)!}\bigg)^2 \frac{1}{2^{4k} k!^2}.
      \end{split}
  \end{equation*}
  Denote the sum on the last line by $s_n$ and use assumption~\eqref{eq:alpha-assumption} to obtain the estimate
  \begin{equation*}
      s_n \leq \sum_{k=0}^{\floor{n/2}} \lambda^{2k} \frac{(n-2k)!}{n!}  \bigg(\frac{n (n-k-1)!}{(n-2k)!}\bigg)^2 \frac{1}{2^{4k} k!^2} = \sum_{k=0}^{\floor{n/2}} b_{n,k} \bigg( \frac{\lambda^2}{16} \bigg)^{k} \frac{1}{k!^2},
\end{equation*}
where
\begin{equation*}
  b_{n, k} = \frac{n (n-k-1)!^2}{(n-1)!(n-2k)!}.
\end{equation*}
It is straightforward to compute that $b_{n, 0} = 1$ for any $n \geq 1$ and $b_{n, 1} = n/(n-1) \leq 2$ for any $n \geq 2$.
Furthermore,
\begin{equation*}
  \frac{b_{n,k}}{b_{n,k+1}} = \frac{(n-k-1)^2}{(n-2k)(n-2k-1)} > 1
\end{equation*}
for any $n \geq 2$ if $k \geq 1$.
Therefore
\begin{equation*}
  s_n \leq 2 \sum_{k=0}^{\floor{n/2}} \bigg( \frac{\lambda^2}{16} \bigg)^{k} \frac{1}{k!^2} \leq 2 e^{\frac{1}{16} \lambda^2 },
\end{equation*}
which gives the claimed upper bound.
The lower bound follows from $s_n \geq 1$.
\end{proof}

\begin{remark} \label{rmk:alpha-assumption-unif}
From the proof of \Cref{lemma:chebyshev-c-lemma} we see that the purpose of assumption~\eqref{eq:alpha-assumption}, which states that
\begin{equation} \label{eq:alpha-assumption-2}
  \frac{\alpha_{k}}{\alpha_n} \leq \lambda^{n-k} \frac{k!}{n!}
\end{equation}
for every $n \geq 0$ and $0 \leq k \leq n$, is to ensure that the sum in~\eqref{eq:chebyshev-norm-lemma} is dominated by the highest order term $\alpha_{n} t_{n, n}^2$.
Suppose that $\alpha_k = k!^\beta$ for some $\beta > 0$ and note that~\eqref{eq:alpha-assumption-2} is satisfied if and only if $\beta \geq 1$.
Then
\begin{equation*}
  \frac{\alpha_{n - 2} t_{n, n - 2}^2}{\alpha_n t_{n, n}^2} = \frac{1}{16} \cdot \frac{n^2}{n^\beta (n-1)^\beta}
\end{equation*}
is bounded from above if and only if $\beta \geq 1$, which shows that for $\beta \in (0, 1)$ the sum in~\eqref{eq:chebyshev-norm-lemma} is \emph{not} dominated by $\alpha_{n} t_{n, n}^2$.
\end{remark}

\begin{lemma}
  \label{lemma:chebyshev-norm-bound}
  The function $f_n(x) = \varphi(x) T_n(x)$ is an element of $H(K)$.
  If assumption~\eqref{eq:alpha-assumption} holds, then
  \begin{equation*}
    \tau_1 2^{2n} \alpha_n \leq \norm[0]{f_n}_K^2 \leq \tau_2 2^{2n} \alpha_n
  \end{equation*}
  for every $n \geq 1$, where the constants $\tau_1$ and $\tau_2$ are given in~\eqref{eq:chebyshev-norm-lemma-constants}.
\end{lemma}
\begin{proof}
  The claim follows from \Cref{lemma:chebyshev-c-lemma} and the Hilbert space characterisation in~\eqref{eq:rkhs-norm} that gives $\norm[0]{f_n}_K^2 = \sum_{k=0}^{\floor{n/2}} \alpha_{n-2k} t_{n, n - 2k}^2$.
\end{proof}

\begin{lemma}
  \label{lemma:unif-norm-bound}
  Let $P_n(x) = \sum_{k=0}^n a_k x^k$ for $a_k \in \R$ be a polynomial of degree $n$.
  The function $f_n(x) = \varphi(x) P_n(x)$ is an element of $H(K)$.
  If assumption~\eqref{eq:alpha-assumption} holds, then
  \begin{equation*}
    \norm[0]{f_n}_K^2 = \sum_{k=0}^n \alpha_k a_k^ 2 \leq (1 + \lambda) \, \tau_2 2^{2n} \alpha_n \norm[0]{P_n}_\infty^2
  \end{equation*}
  for every $n \geq 1$, where the constant $\tau_2$ is given in~\eqref{eq:chebyshev-norm-lemma-constants}.
\end{lemma}
\begin{proof}
  Because the polynomial $Q_n = P_n / \lVert P_n \rVert_\infty$ has the coefficients $a_k / \lVert P_n \rVert_\infty$, we may assume that $\norm[0]{P_n}_\infty \leq 1$.
  In particular, $\abs[0]{P_n(\tilde{x}_{n,k})} \leq 1$ for every $0 \leq k \leq n$, where $\tilde{x}_{n,k}$ are extremal points of $T_n$ in~\eqref{eq:chebyshev-extrema}.
  We get from~\eqref{eq:general-markov-inequality} that
  \begin{equation*}
    \abs[0]{a_{n-2k}} + \abs[0]{a_{n-2k-1}} \leq \abs[0]{t_{n,n-2k}} \quad \text{ for every } \quad 0 \leq k \leq \floor[\bigg]{\frac{n-1}{2}}.
  \end{equation*}
  Furthermore, it is obvious that $\abs[0]{a_0} \leq \norm[0]{P_n}_\infty \leq t_{n, 0} = 1$ when $n$ is even.
  These inequalities and the characterisation~\eqref{eq:rkhs-norm} give
  \begin{equation*}
    \begin{split}
      \norm[0]{f_n}_K^2 = \sum_{k=0}^n \alpha_k a_k^2 &= \sum_{k=0}^{\floor{n/2}} \alpha_{n - 2k} a_{n - 2k}^2 + \sum_{k=0}^{\floor{(n-1)/2}} \alpha_{n - 2k -1} a_{n - 2k - 1}^2 \\
      &\leq \sum_{k=0}^{\floor{n/2}} \alpha_{n - 2k} t_{n, n - 2k}^2 + \sum_{k=0}^{\floor{(n-1)/2}} \alpha_{n - 2k -1} t_{n, n - 2k}^2.
      \end{split}
  \end{equation*}
  From assumption~\eqref{eq:alpha-assumption} we get
  \begin{equation*}
      \sum_{k=0}^{\floor{(n-1)/2}} \alpha_{n - 2k -1} t_{n, n - 2k}^2 \leq \lambda \sum_{k=0}^{\floor{(n-1)/2}} \frac{1}{n-2k} \alpha_{n - 2k} t_{n, n - 2k}^2 \leq \lambda \sum_{k=0}^{\floor{n/2}} \alpha_{n - 2k} t_{n, n - 2k}^2,
  \end{equation*}
  so that the claim follows from \Cref{lemma:chebyshev-c-lemma}.
\end{proof}

With these tools at hand we are ready to bound the worst-case errors from below.

\begin{proof}[Proofs of lower bounds in \Cref{thm:main-3} for $p=\infty$]
Let $P_n(x) = (x - x_1) \cdots (x - x_n)$.
This is a polynomial of degree $n$ that vanishes at the points $x_1, \ldots, x_n \in [-1, 1]$.
The function $f_n(x) = \varphi(x) P_n(x)$ is an element of $H(K)$ that vanishes at the same points.
Therefore $\sum_{k=1}^n f_n(x_k) \psi_k \equiv 0$ and thus 
\begin{equation*}
        e_\infty^{\min}(x_1, \ldots, x_n) = \inf_{\psi_1, \ldots, \psi_n \in L^\infty} \sup_{0 \neq f \in H(K)} \frac{ \lVert f - \sum_{k=1}^n f(x_k) \psi_k \rVert_\infty }{\lVert f \rVert_{K}} \geq \frac{\lVert f_n \rVert_\infty}{\lVert f_n \rVert_{K}} .
\end{equation*}
  The lower bounds in~\eqref{eq:main-31} and~\eqref{eq:main-32} for $p = \infty$ now follow from \Cref{lemma:unif-norm-bound} and
  \begin{equation*}
    \norm[0]{f_n}_\infty = \sup_{x \in [-1, 1]} \abs[0]{\varphi(x) P_n(x)} \geq \varphi_\textup{min} \norm[0]{P_n}_\infty.
  \end{equation*}
Observe that $\lVert P_n \rVert_\infty \leq 2^n$ since $\lvert x - y \lvert \leq 2$ for any $x, y \in [-1, 1]$.
Therefore
\begin{equation*}
        e^{\min}(x \mid x_1, \ldots, x_n) \geq \frac{\lvert f_n(x) \rvert}{\lVert f_n \rVert_K} \geq \frac{\varphi_\textup{min}}{\lVert f_n \rVert_K} \bigg\lvert \prod_{k=1}^n (x - x_n) \bigg\rvert \geq \frac{\varphi_\textup{min}\lVert P_n \rVert_\infty }{2^n \lVert f_n \rVert_K} \bigg\lvert \prod_{k=1}^n (x - x_n) \bigg\rvert,
\end{equation*}
so that the lower bound in \eqref{eq:main-33} follows from \Cref{lemma:unif-norm-bound}.
\end{proof}

\subsection{Lower bounds for approximation in $L^2([-1, 1])$} \label{sec:lower-L2}

In this section we prove lower bounds on $e_2^{\min}(x_1, \ldots, x_n)$ and $e_2^{\min}(n)$.
The technique is identical to the one that we used for $p = \infty$.
Define
\begin{equation} \label{eq:lnk-definition}
  l_{n,k} = \frac{(2k-1)!!}{k!} \sqrt{\smash[b]{k+\tfrac{1}{2}}} \binom{\floor{(n-k)/2} + k + \frac{1}{2}}{\floor{(n-k)/2}}
\end{equation}
for any $0 \leq k \leq n$.
Here $(2k-1)!! = (2k-1)(2k-3) \cdots 3 \cdot 1$ is the double factorial and 
\begin{equation*}
  \binom{r}{k} = \frac{r(r-1) \cdots (r-k+1)}{k!}
\end{equation*}
the generalised binomial coefficient defined for any $r \in \R$ and $k \in \N_0$. 
In this section the constants $l_{n,k}$ play an analogous role to that played by $t_{n, n - 2k}$ in \Cref{sec:lower-Linf} because Labelle~\cite{Labelle1969} has proved an $L^2$-version of the polynomial coefficient bound~\eqref{eq:general-markov-inequality} that involves $l_{n,k}$.
Namely, if $P_n(x) = \sum_{k=0}^n a_k x^k$ is a polynomial of degree $n$ with real coefficients $a_k$, then
\begin{equation} \label{eq:labelle-bound}
  \abs[0]{a_k} \leq l_{n,k} \norm[0]{P_n}_{2} \quad \text{ for every } \quad 0 \leq k \leq n.
\end{equation}
Equality in~\eqref{eq:labelle-bound} holds only for certain sums of Legendre polynomials, but we shall not use this fact.
See also~\cite[pp.\@~676--7]{RahmanSchmeisser2002}.

\begin{lemma} \label{lemma:legendre-c-lemma}
  Suppose that~\eqref{eq:alpha-assumption} holds.
  Then for every $n \geq 1$ we have
  \begin{equation*}
    \ell_1 2^{2n} \alpha_n \leq \sum_{k=0}^n \alpha_k l_{n,k}^2 \leq \ell_2 2^{2n} \alpha_n,
  \end{equation*}
  where
  \begin{equation} \label{eq:c-ell-L2}
    \ell_1 = \frac{1}{2\pi e^2} \quad \text{ and } \quad \ell_2 = \frac{27e^6}{2\pi} \bigg(1 + \frac{\lambda}{4} \bigg) e^{\frac{1}{16} \lambda^2}.
  \end{equation}
\end{lemma}
\begin{proof}
  For $k \geq 1$, the non-asymptotic Stirling's formula,
\begin{equation} \label{eq:stirling}
  \sqrt{2\pi} \, n^{n+1/2} e^{-n} < n! < \sqrt{2\pi} \, n^{n+1/2} e^{-n+1} ,
\end{equation}
and the double factorial identity $(2k-1)!! = (2k)!/(2^k k!)$ yield
  \begin{equation} \label{eq:d-fact-upper}
    \bigg( \frac{(2k-1)!!}{k!}\bigg)^2 (k+\tfrac{1}{2}) = \frac{(2k)!^2}{2^{2k} k!^4} (k+\tfrac{1}{2}) \leq 2^{2k} \frac{e^2}{\pi} \cdot \frac{k+\frac{1}{2}}{k} \leq \frac{3e^2}{2\pi} \cdot 2^{2k} 
  \end{equation}
  and
  \begin{equation} \label{eq:d-fact-lower}
    \bigg( \frac{(2k-1)!!}{k!}\bigg)^2 (k + \tfrac{1}{2}) \geq 2^{2k} \frac{1}{\pi e^2} \cdot \frac{k+\frac{1}{2}}{k} \geq \frac{1}{2\pi e^2} \cdot 2^{2k}.
  \end{equation}
  It is trivial to verify that these inequalities hold also for $k = 0$.
  The definition of $l_{n,k}$ in~\eqref{eq:lnk-definition} thus gives
  \begin{equation*}
   \frac{1}{2\pi e^2} \cdot 2^{2k} \binom{\floor{(n-k)/2} + k + \frac{1}{2}}{\floor{(n-k)/2}}^2 \leq l_{n, k}^2 \leq \frac{3e^2}{2\pi} \cdot 2^{2k} \binom{\floor{(n-k)/2} + k + \frac{1}{2}}{\floor{(n-k)/2}}^2.
  \end{equation*}
  Note that the generalised binomial coefficient equals one when $\floor{(n-k)/2} = 0$, so that
  \begin{equation*}
    \frac{1}{2\pi e^2} \cdot 2^{2n} \leq l_{n, n}^2 = \bigg( \frac{(2n-1)!!}{n!}\bigg)^2 (n + \tfrac{1}{2}) \leq \frac{3e^2}{2\pi} \cdot 2^{2n}
  \end{equation*}
  by~\eqref{eq:d-fact-upper} and~\eqref{eq:d-fact-lower} and $k + \tfrac{1}{2}$
  We now get
  \begin{equation*}
    \sum_{k=0}^n \alpha_k l_{n,k}^2 \geq \alpha_n l_{n,n}^2 \geq \frac{1}{2\pi e^2} \cdot 2^{2n} \alpha_n,
  \end{equation*}
  which is the claimed lower bound, and
  \begin{equation*}
    \begin{split}
    \sum_{k=0}^n \alpha_k l_{n,k}^2 &= \alpha_n l_{n,n}^2 \sum_{k=0}^n \frac{\alpha_k}{\alpha_n} \bigg( \frac{l_{n,k}}{l_{n,n}} \bigg)^2 \\
    &\leq \frac{3e^2}{2\pi} \cdot 2^{2n} \alpha_n \sum_{k=0}^n \frac{\alpha_k}{\alpha_n} \cdot \frac{3e^2}{2\pi} \cdot 2\pi e^2 \cdot 2^{-2(n-k)} \binom{\floor{(n-k)/2} + k + \frac{1}{2}}{\floor{(n-k)/2}}^2 \\
    &= \frac{9e^6}{2\pi} \cdot 2^{2n} \alpha_n \sum_{k=0}^n \frac{\alpha_k}{\alpha_n} \cdot 2^{-2(n-k)} \binom{\floor{(n-k)/2} + k + \frac{1}{2}}{\floor{(n-k)/2}}^2.
    \end{split}
  \end{equation*}
  Assumption~\eqref{eq:alpha-assumption} therefore gives
  \begin{equation} \label{eq:s-l2-1}
    \sum_{k=0}^n \alpha_k l_{n,k}^2 \leq \frac{9e^6}{2\pi} \cdot 2^{2n} \alpha_n \sum_{k=0}^n \bigg( \frac{\lambda}{4} \bigg)^{n-k} \frac{k!}{n!} \binom{\floor{(n-k)/2} + k + \frac{1}{2}}{ \floor{(n-k)/2}}^2.
  \end{equation}
  We now show that the term
  \begin{equation*}
    b_{n, k} = \frac{k!}{n!} \binom{\floor{(n-k)/2} + k + \frac{1}{2}}{\floor{(n-k)/2}}^2
  \end{equation*}
  is approximately $(n - k)!$.
  Denote $p_{n,k} = \floor{(n-k)/2}$ and suppose that $p_{n,k} \geq 1$.
  Then
  \begin{align*}
      b_{n, k} &= \frac{k!}{n!} \cdot \frac{\prod_{\nu = 1}^{p(n,k)} (\nu + k + \frac{1}{2})^2}{p(n,k)!^2} \\
    &= \frac{1}{ p(n,k)!^2 } \cdot \frac{\prod_{\nu = 1}^{p(n,k)} (\nu + k + \frac{1}{2})^2}{\prod_{\nu=1}^{n-k} (\nu + k)}. \\
          &= \frac{1}{ p(n,k)!^2 } \cdot \frac{\prod_{\nu = 1}^{p(n,k)} (\nu + k + \frac{1}{2})}{\prod_{\nu=1}^{p(n,k)} (\nu + k)} \cdot \frac{\prod_{\nu = 1}^{p(n,k)} (\nu + k + \frac{1}{2})}{\prod_{\nu=p(n,k)+1}^{n-k} (\nu + k)} \\
    &= \frac{1}{ p(n,k)!^2 } \cdot \frac{\prod_{\nu = 2}^{p(n,k)} (\nu + k - \frac{1}{2})}{\prod_{\nu=2}^{p(n,k)} (\nu + k)} \cdot \frac{p(n,k) + k + \frac{1}{2}}{k + 1} \cdot \frac{\prod_{\nu = 1}^{p(n,k)} (\nu + k + \frac{1}{2})}{\prod_{\nu=p(n,k)+1}^{n-k} (\nu + k)} \\
      &\leq \frac{1}{ p(n,k)!^2 } \cdot \frac{p(n,k) + k + \frac{1}{2}}{k + 1} \cdot \frac{\prod_{\nu = 1}^{p(n,k)} (\nu + k + \frac{1}{2})}{\prod_{\nu=p(n,k)+1}^{n-k} (\nu + k)} \\
      &= \frac{1}{ p(n,k)!^2 } \cdot \frac{p(n,k) + k + \frac{1}{2}}{k + 1} \cdot \frac{k + \frac{3}{2}}{n} \cdot \frac{\prod_{\nu = 2}^{p(n,k)} (\nu + k + \frac{1}{2})}{\prod_{\nu=p(n,k)+1}^{n-k-1} (\nu + k)} \\
      &\leq \frac{3}{ p(n,k)!^2 } \cdot \frac{\prod_{\nu = 2}^{p(n,k)} (\nu + k + \frac{1}{2})}{\prod_{\nu=p(n,k)+1}^{n-k-1} (\nu + k)},
  \end{align*}
  where the last inequality is a consequence of the inequalities $(p(n,k) + k + \frac{1}{2})/n \leq \frac{3}{2}$ and $(k + \frac{3}{2}) / (k+1) \leq 2$.
  Now, in the second term on the last line, each term in the denominator is larger than any term in the numerator and there are at least as many terms in the denominator as there are in the numerator.
  Therefore,
  \begin{equation} \label{eq:b-estimate-l2}
    b_{n, k} \leq \frac{3}{ p(n,k)!^2 } = \frac{3}{\floor{(n-k)/2}!^2}
  \end{equation}
  if $n-k \geq 2$. By computing $b_{n,k}$ in the cases $k = n$ and $k = n - 1$ it is trivial to check that this inequality is valid for any $0 \leq k \leq n$.
  Using the estimate~\eqref{eq:b-estimate-l2} in~\eqref{eq:s-l2-1} yields
  \begin{equation*}
    \begin{split}
      \sum_{k=0}^n \alpha_k l_{n,k}^2 \leq \frac{9e^6}{2\pi} \cdot 2^{2n} \alpha_n \sum_{k=0}^n \bigg( \frac{\lambda}{4} \bigg)^{n-k} b_{n,k} &\leq \frac{27e^6}{2\pi} \cdot 2^{2n} \alpha_n \sum_{k=0}^n \bigg( \frac{\lambda}{4} \bigg)^{k} \frac{1}{\floor{k/2}!^2} \\
      &\leq \frac{27e^6}{2\pi} \cdot 2^{2n} \alpha_n \sum_{k=0}^n \bigg( \frac{\lambda}{4} \bigg)^{2k} \bigg(1 + \frac{\lambda}{4} \bigg) \frac{1}{k!^2} \\
      &\leq \frac{27e^6}{2\pi} \bigg(1 + \frac{\lambda}{4} \bigg) e^{\frac{1}{16} \lambda^2} \cdot 2^{2n} \alpha_n.
    \end{split}
  \end{equation*}
  This completes the proof.
\end{proof}

\begin{remark} \label{rmk:alpha-assumption-l2}
  Just as the proof of \Cref{lemma:chebyshev-c-lemma} (see \Cref{rmk:alpha-assumption-unif}), the proof of \Cref{lemma:legendre-c-lemma} is based on the fact that assumption~\eqref{eq:alpha-assumption} ensures that the term $\alpha_n l_{n,n}^2$ dominates the sum in~\eqref{eq:chebyshev-norm-lemma}.
  Suppose that $\alpha_k = k!^\beta$ for some $\beta > 0$ and note that~\eqref{eq:alpha-assumption} is satisfied if and only if $\beta \geq 1$.
  Then Stirling's formula yields
  \begin{equation*}
    \begin{split}
      \frac{\alpha_{n-2} l_{n, n-2}^2}{\alpha_n l_{n,n}^2} &= \bigg( \frac{(n-2)!}{n!} \bigg)^\beta \bigg( \frac{(2n-5)!! \, n!}{(2n-1)!! \, (n-2)!} \bigg)^2 \frac{n - \frac{3}{2}}{n + \frac{1}{2}} \binom{n - \frac{1}{2}}{1}^2 \\
      &\sim \bigg( \frac{1}{n(n-1)} \bigg)^\beta \frac{1}{16} \cdot \frac{n}{n-2} \cdot \frac{n - \frac{3}{2}}{n + \frac{1}{2}} \cdot (n - \tfrac{1}{2})^2 \\
      &\sim \frac{1}{16} \cdot n^{2(1 - \beta)}
      \end{split}
  \end{equation*}
  as $n \to \infty$.
  That is, the ratio is bounded from above if and only if $\beta \geq 1$, which shows that for $\beta \in (0, 1)$ the sum in~\eqref{eq:chebyshev-norm-lemma} is \emph{not} dominated by $\alpha_{n} l_{n, n}^2$.
\end{remark}

\begin{lemma} \label{lemma:l2-norm-bound}
  Let $P_n(x) = \sum_{k=0}^n a_k x^k$ for $a_k \in \R$ be a polynomial of degree $n$.
  The function $f_n(x) = \varphi(x) P_n(x)$ is an element of $H(K)$.
  If assumption~\eqref{eq:alpha-assumption} holds, then
  \begin{equation*}
    \norm[0]{f_n}_K^2 = \sum_{k=0}^n \alpha_k a_k^ 2 \leq \ell_2 2^{2n} \alpha_n \norm[0]{P_n}_{2}^2
  \end{equation*}
  for every $n \geq 1$, where the constant $\ell_2$ is given in~\eqref{eq:c-ell-L2}.
\end{lemma}
\begin{proof}
  The characterisation~\eqref{eq:rkhs-norm} and inequality~\eqref{eq:labelle-bound} give
  \begin{equation*}
    \norm[0]{f_n}_K^2 = \sum_{k=0}^n \alpha_k a_k^2 \leq \norm[0]{P_n}_{2}^2 \sum_{k=0}^n \alpha_k l_{n,k}^2.
  \end{equation*}
  The claim then follows from \Cref{lemma:legendre-c-lemma}.
\end{proof}

With these tools at hand we are ready to bound the worst-case errors from below.

\begin{proof}[Proofs of lower bounds in \Cref{thm:main-3} for $p=2$]
The proof is identical to that for the case $p = \infty$, except that we now use \Cref{lemma:l2-norm-bound} to estimate $\lVert f_n \rVert_K$.
\end{proof}

\subsection{Upper bounds} \label{sec:upper-bounds}

To bound the worst-case errors from above we shall use weighted polynomial interpolation.
Upper bounds on derivatives of elements of $H(K)$ are needed for this purpose.
The following results closely resemble certain results in scattered data approximation literature, particularly in~\cite{Zwicknagl2009}.
Define
\begin{equation} \label{eq:CK2n-definition}
  C_K^{2n} = \sup_{x \in [-1, 1]} \frac{\partial^{2n}}{\partial v^n \partial w^n} K(v, w) \biggl|_{\substack{v = x \\ w = x}}.
\end{equation}

\begin{lemma} \label{lemma:rkhs-norm-derivative}
  Let $K$ be any positive-semidefinite kernel on $[-1, 1] \times [-1, 1]$ that is infinitely differentiable in both arguments.
  Then every $f \in H(K)$ is infinitely differentiable and
  \begin{equation*}
    \norm[0]{f^{(n)}}_\infty \leq \norm[0]{f}_K (C_K^{2n})^{1/2}
  \end{equation*}
  for every $n \geq 0$.
\end{lemma}
\begin{proof}
  It is well known that, for a kernel $K$ which is $n$ times differentiable with respect to both of its arguments, every $f \in H(K)$ is $n$ times differentiable and the reproducing property $\inprod{f}{K(x, \cdot)}_K = f(x)$ extends to differentiation~\cite[Cor.\@~4.36]{Steinwart2008}.
  The Cauchy--Schwarz inequality thus gives
  \begin{equation*}
    \abs[0]{f^{(n)}(x)} = \abs[3]{ \inprod[\bigg]{f}{\frac{\partial^n}{\partial x^n} K(x, \cdot)}_K } \leq \norm[0]{f}_K \, \norm[3]{ \frac{\partial^n}{\partial x^n} K(x, \cdot)}_K
  \end{equation*}
  for every $x \in [-1, 1]$.
  Again by the derivative reproducing property,
  \begin{equation*}
    \norm[3]{ \frac{\partial^n}{\partial x^n} K(x, \cdot)}_K^2 = \frac{\partial^{2n}}{\partial v^n \partial w^n} K(v, w) \biggl|_{\substack{v = x \\ w = x}}.
  \end{equation*}
  The claim follows.
\end{proof}

The Laguerre polynomial of degree $n \geq 0$ is defined as
\begin{equation} \label{eq:laguerre-polynomial}
  L_n(x) = \sum_{k=0}^n \binom{n}{k} \frac{(-x)^k}{k!} = \frac{e^x}{n!} \cdot \frac{\dif^{\,n}}{\dif x^n} (e^{-x} x^n).
\end{equation}
We make use of the following effective asymptotics for Laguerre polynomials from~\cite{Borwein2008}.

\begin{proposition} \label{prop:laguerre}
  Let $x > 0$.
  Define
  \begin{equation*}
    \tilde{c}_L(x) = \frac{e^{-\frac{1}{2} x + 2\sqrt{\smash[b]{x}}}}{2\sqrt{\pi} x^{1/4}} \bigg( 3 + \sqrt{\frac{\pi}{e}} + \frac{2}{\sqrt{\pi} x^{1/4}} + \frac{7}{16 \sqrt{x}} + \sqrt{x} (2 + x) \bigg)
  \end{equation*}
  and
  \begin{equation*}
    \tilde{m}_L(x) = \max\bigg\{ 5x, x\bigg(1+\frac{x}{2} \bigg)^2, \frac{9}{x} \bigg\}.
  \end{equation*}
  Then
  \begin{equation} \label{eq:laguerre-bound}
    L_n(-x) = \sum_{k=0}^n \binom{n}{k} \frac{x^k}{k!} \leq \tilde{c}_L(x) \, n^{-1/4} e^{2\sqrt{\smash[b]{x n}}}
  \end{equation}
  for every $n \geq \tilde{m}_L(x)$.
\end{proposition}
\begin{proof}
  Setting $a = 0$ in Section~6.3 of \cite{Borwein2008} gives
  \begin{equation} \label{eq:laguerre-bound-initial}
    \begin{split}
      L_n(-x) &\leq \frac{e^{-\frac{1}{2} x}}{2\sqrt{\pi}} \cdot \frac{e^{2\sqrt{\smash[b]{x (n+1)}}}}{x^{1/4} (n + 1)^{1/4}}\bigg(1 + \frac{\overline{C}_1}{\sqrt{n+1}} + \mathcal{E}_1 + \mathcal{E}_{3,1} \bigg) \\
      &\leq \frac{e^{-\frac{1}{2} x + 2\sqrt{\smash[b]{x}}}}{2\sqrt{\pi}} \cdot \frac{e^{2\sqrt{\smash[b]{x n}}}}{x^{1/4} \, n^{1/4}} \, (1 + \overline{C}_1 + \mathcal{E}_1 + \mathcal{E}_{3,1} )
      \end{split}
  \end{equation}
  if $n \geq \max\{m_0, m_1, \ldots, m_7\}$, where $\overline{C}_1$, $\mathcal{E}_1$, $\mathcal{E}_{3,1}$ and $m_0, m_1, \ldots, m_7$ are certain non-negative constants.
  From the forms given for the constants $m_i$ on p.\@~3303 we see that~\eqref{eq:laguerre-bound-initial} holds if
  \begin{equation*}
    \begin{split}
      n \geq \max\{m_0, m_1, \ldots, m_7\} &= \max\bigg\{ \frac{x}{4}, 5x, 0, -\frac{5}{4} x, 4x, x\bigg(1+ \frac{x}{2} \bigg)^2, \frac{9}{x}, 0 \bigg\} \\
      &= \max\bigg\{ 5x, x\bigg(1+\frac{x}{2} \bigg)^2, \frac{9}{x} \bigg\}.
      \end{split}
  \end{equation*}
  The constants $\overline{C}_1$, $\mathcal{E}_1$ and $\mathcal{E}_{3,1}$ satisfy (see pp.\@~3306, 3304 and 3294, respectively)
  \begin{align*}
    \overline{C}_1 &\leq \frac{7}{16 \sqrt{x}} + 2 \sqrt{x} (1 + x / 2) , \\
    \mathcal{E}_1 &\leq 2 e^{-4 \sqrt{\smash[b]{x(n+1)}}} + 2\sqrt{\pi} x^{1/4} (n+1)^{1/4} e^{-2 \sqrt{\smash[b]{x(n+1)}}}, \\
    \mathcal{E}_{3, 1} &\leq \frac{2 e^{-2 \sqrt{\smash[b]{x(n+1)}}} }{\sqrt{\pi} x^{1/4} (n+1)^{1/4}}.
  \end{align*}
  The last two of these we may further bound as
  \begin{equation*}
    \mathcal{E}_1 \leq 2 + 2\sqrt{\pi} x^{1/4} (n+1)^{1/4} e^{-2 \sqrt{\smash[b]{x (n+1)}}} \leq 2 + \sqrt{\frac{\pi}{e}} \quad \text{ and } \quad \mathcal{E}_{3,1} \leq \frac{2}{\sqrt{\pi} x^{1/4}},
  \end{equation*}
  where we have used the inequality $2y^{1/4} e^{-2 \sqrt{y}} \leq e^{-1/2}$ for $y \geq 0$.
\end{proof}

Note that the rate in~\eqref{eq:laguerre-bound} cannot be improved because~\eqref{eq:laguerre-bound-initial} is, in fact, not only an upper bound but an asymptotic equivalence for $L_n(-x)$.
The following proposition is similar to the lemmas (in particular Lemma~3) in Section~5 of~\cite{Zwicknagl2009}.\footnote{It seems that the eigenvalue bound used in the proof of Lemma~3 in~\cite{Zwicknagl2009} to estimate $L_n(-1)$ is erroneous because a more careful argument for bounding the log-sum gives
  \begin{equation*}
    \sum_{k=1}^{n-1} \log \bigg(1 + \frac{1}{\sqrt{k+1}} \bigg) \leq \int_1^n \log \bigg(1 + \frac{1}{\sqrt{x}} \bigg) \dif x = \sqrt{n} + (n-1)\log\bigg( 1 + \frac{1}{n}\bigg) - \frac{1}{2} \log n - 1,
  \end{equation*}
  which implies that $L_n(-1) = O(n^{-1/2} e^{2\sqrt{n}})$.
  This contradicts the Laguerre asymptotics in~\cite{Borwein2008}.
  Although its proof is erroneous, the upper bound $L_n(-1) \leq 2 e^{2\sqrt{n}}$ derived in the proof of the lemma in question is correct (at least for sufficiently large $n$) in the light of~\cite{Borwein2008} and our \Cref{prop:laguerre}.
}
See also~\cite[Sec.\@~6.4]{ZwicknaglSchaback2013}.
Define the unweighted kernel
\begin{equation} \label{eq:K0-kernel}
  R(x, y) = \sum_{k=0}^\infty \alpha_k^{-1} x^k y^k .
\end{equation}

\begin{lemma} \label{lemma:CK-2n-bound}
  Suppose that assumption~\eqref{eq:alpha-assumption} holds and let $R$ be the unweighted kernel in~\eqref{eq:K0-kernel}.
  Then
  \begin{equation*}
    C_{R}^{2n} \leq c_L(\lambda) \, n^{-1/4} e^{2\sqrt{\smash[b]{\lambda n}}} \alpha_n^{-1} n!^2
  \end{equation*}
  for every $n \geq m_L(\lambda)$, where
  \begin{equation} \label{eq:cL-constant}
    c_L(\lambda) = \frac{e^{\frac{1}{2} \lambda + 2\sqrt{\smash[b]{\lambda}}}}{2\sqrt{\pi} \lambda^{1/4}} \bigg( 3 + \sqrt{\frac{\pi}{e}} + \frac{2}{\sqrt{\pi} \lambda^{1/4}} + \frac{7}{16 \sqrt{\lambda}} + \sqrt{\lambda} (2 + \lambda) \bigg)
  \end{equation}
  and
  \begin{equation} \label{eq:mL-constant}
    m_L(\lambda) = \max\bigg\{ 5\lambda, \lambda\bigg(1+\frac{\lambda}{2} \bigg)^2, \frac{9}{\lambda} \bigg\}.
  \end{equation}
\end{lemma}
\begin{proof}
  We compute
  \begin{equation*}
      C_{R}^{2n} = \sup_{x = y \in [-1, 1]} \sum_{k=n}^\infty \alpha_k^{-1} \bigg(\frac{k!}{(k - n)!}\bigg)^2 x^{k-n} y^{k-n} = \sum_{k=n}^\infty \alpha_k^{-1} \bigg(\frac{k!}{(k - n)!}\bigg)^2.
  \end{equation*}
  Assumption~\eqref{eq:alpha-assumption} yields
  \begin{equation} \label{eq:CK2n-definition}
    \begin{split}
      C_{R}^{2n} = \alpha_n^{-1} \sum_{k=n}^\infty \frac{\alpha_n}{\alpha_k} \bigg(\frac{k!}{(k - n)!}\bigg)^2 &= \alpha_n^{-1} \sum_{k=0}^\infty \frac{\alpha_n}{\alpha_{n+k}} \bigg(\frac{(n + k)!}{k!}\bigg)^2 \\
      &\leq \alpha_n^{-1} \sum_{k=0}^\infty \lambda^k \frac{n!}{(n + k)!} \bigg(\frac{(n + k)!}{k!}\bigg)^2 \\
      &= \alpha_n^{-1} n! \sum_{k=0}^\infty \lambda^k \frac{(n + k)!}{k!^2}.
      \end{split}
  \end{equation}
  From these computations it is easy to see that (set $\alpha_k = \lambda^{-k} k!$ so that each inequality becomes an equality)
  \begin{equation} \label{eq:exp-kernel-CK2n-proof-1}
    \frac{\partial^{2n}}{\partial x^n \partial y^n} e^{\lambda xy} \biggl|_{\substack{x = 1 \\ y = 1}} = \lambda^n \sum_{k=0}^\infty \lambda^k \frac{(n + k)!}{k!^2}.
  \end{equation}
  However, by~\eqref{eq:laguerre-polynomial} we have an alternative expression for the above derivative in terms of Laguerre polynomials:
  \begin{equation} \label{eq:exp-kernel-CK2n-proof-2}
    \frac{\partial^{2n}}{\partial x^n \partial y^n} e^{\lambda xy} \biggl|_{\substack{x = 1 \\ y = 1}} = \frac{\dif^{\, n}}{\dif y^n} \lambda^n y^n e^{\lambda y} \biggl|_{y = 1} = e^{\lambda} \lambda^n n! L_n(-\lambda).
  \end{equation}
  Combining Equations~\eqref{eq:CK2n-definition}--\eqref{eq:exp-kernel-CK2n-proof-2} therefore yields
  \begin{equation*}
    C_R^{2n} \leq e^{\lambda} L_n(-\lambda) \alpha_n^{-1} n!^2.
  \end{equation*}
  The claim follows from \Cref{prop:laguerre}.
\end{proof}

It is useful, though not necessary for what follows, to observe that assumption~\eqref{eq:alpha-assumption} and the analyticity $\varphi$ ensure that $H(K)$ consists of analytic functions.
We refer to~\cite[pp.\@~40--43]{Saitoh1997} and~\cite{SaitohSawano2016, SunZhou2008} for general discussion and results on analyticity in the context of reproducing kernel Hilbert spaces.
Results on analyticity and related properties for kernels defined via Hermite polynomials may be found in~\cite[Sec.\@~3.1]{Gnewuch2022} and references therein.

\begin{proposition} \label{prop:analytic}
  Suppose that assumption~\eqref{eq:alpha-assumption} holds.
  If $\varphi$ is analytic, then every element of $H(K)$ is analytic.
\end{proposition}
\begin{proof}
  Without loss of generality we can assume that $\alpha_0 = 1$.
  Then assumption~\eqref{eq:alpha-assumption} implies that $\alpha_n \geq \lambda^{-n} n!$.
  If $f \in H(R)$ for the unweighted kernel $R$ in~\eqref{eq:K0-kernel}, Lemmas~\ref{lemma:rkhs-norm-derivative} and~\ref{lemma:CK-2n-bound} yield
  \begin{equation*}
    \begin{split}
      \norm[0]{f^{(n)}}_\infty \leq \norm[0]{f}_{R} C^n n! \, \alpha_n^{-1/2} \leq \norm[0]{f}_{R} C^n n! \sqrt{ \frac{ \lambda^n }{n!}} &\leq \norm[0]{f}_{R} (C \lambda^{1/2} \, )^n \sqrt{n!} \\
      &\leq \norm[0]{f}_{R} (C \lambda^{1/2} \, )^n n!
      \end{split}
  \end{equation*}
  for a certain positive $C$, which implies that $f$ is analytic~\cite[Lem.\@~1.2.10]{KrantzParks2002}.
  Since the mapping $f \mapsto \varphi f$ is an isometric isomorphism between $H(R)$ and $H(K)$, we see that if $\varphi$ is analytic, every element of $H(K)$ is analytic, being a product of two analytic functions.
\end{proof}

Let $f \colon I \to \R$ be an $n$ times continuously differentiable function on a closed interval $I$.
The unique polynomial $S_n f$ of degree $n-1$ that interpolates $f$ at some distinct points $x_1, \ldots, x_n \in I$ can be written in terms of the Lagrange basis functions as
\begin{equation} \label{eq:polynomial-interpolant}
  (S_{n} f)(x) = \sum_{k=1}^n f(x_k) \prod_{i \neq k} \frac{x - x_i}{x_k - x_i}.
\end{equation}
Define the \emph{weighted polynomial interpolant}
\begin{equation} \label{eq:polynomial-interpolant-weighted}
  (S_{n}^\varphi f)(x) = \varphi(x) \sum_{k=1}^n \frac{f(x_k)}{\varphi(x_k)} \prod_{i \neq k} \frac{x - x_i}{x_k - x_i}.
\end{equation}
Both of these interpolants have the form $\sum_{k=1}^n f(x_k) \psi_k$ for certain functions $\psi_k$.
It is a standard result~\cite[Sec.\@~2.6]{Hildebrand1987} that for each $x \in I$ there exists $\xi_x \in I$ such that
\begin{equation} \label{eq:polynomial-interpolation-error}
  f(x) - (S_{n} f)(x) = \frac{f^{(n)}(\xi_x)}{n!} \prod_{k=1}^n (x - x_k).
\end{equation}
Upper bounds on the worst-case errors will be straightforward corollaries of the following proposition.

\begin{proposition} \label{prop:poly-interp-bound}
  Suppose that assumption~\eqref{eq:phi-assumption} holds.
  If $f \in H(K)$ and the points $x_1, \ldots, x_n \in [-1, 1]$ are distinct, then
  \begin{equation} \label{eq:poly-interp-bound-1}
    \abs[0]{ f(x) - (S_{n}^\varphi f)(x) } \leq \varphi_{\max} \norm[0]{f}_K \frac{(C_{R}^{2n})^{1/2}}{n!} \, \abs[3]{ \prod_{k=1}^n (x - x_k) }
  \end{equation}
  for every $x \in [-1, 1]$.
\end{proposition}
\begin{proof}
  Since the mapping $f \mapsto \varphi f$ is an isometric isomorphism, for every $f \in H(K)$ there is $g \in H(R)$ such that $f = \varphi g$ and $\norm[0]{f}_K = \norm[0]{g}_{R}$.
  The definitions of $S_{n}$ and $S_{n}^\varphi$ and the assumption~\eqref{eq:phi-assumption} that $\varphi$ is bounded from above give
  \begin{equation*}
    \abs[0]{ f(x) - (S_{n}^\varphi f)(x) } = \abs[0]{\varphi(x)} \abs[0]{ g(x) - (S_{n} g)(x) } \leq \varphi_{\max} \abs[0]{ g(x) - (S_{n} g)(x) }.
  \end{equation*}
  By Equation~\eqref{eq:polynomial-interpolation-error} and Lemma~\ref{lemma:rkhs-norm-derivative},
  \begin{equation*}
    \begin{split}
    \abs[0]{ g(x) - (S_{n} g)(x) } \leq \frac{\norm[0]{g^{(n)}}_\infty}{n!} \, \abs[3]{ \prod_{k=1}^n (x - x_k) } &\leq \norm[0]{g}_{R} \frac{(C_{R}^{2n})^{1/2}}{n!} \, \abs[3]{ \prod_{k=1}^n (x - x_k) } \\
    &= \norm[0]{f}_{K} \frac{(C_{R}^{2n})^{1/2}}{n!} \, \abs[3]{ \prod_{k=1}^n (x - x_k) },
    \end{split}
  \end{equation*}
  which gives the claim.
\end{proof}

Note that the above proposition does not require that $\varphi$ be differentiable.
If $\varphi$ is assumed to be infinitely differentiable and $C_K^{2n}$ can be computed or estimated, one can use the polynomial interpolant~\eqref{eq:polynomial-interpolant} and the estimate
\begin{equation} \label{eq:poly-interp-bound-2}
  \abs[0]{ f(x) - (S_{n} f)(x) } \leq \norm[0]{f}_K \frac{(C_{K}^{2n})^{1/2}}{n!} \, \abs[3]{ \prod_{k=1}^n (x - x_k) }
\end{equation}
instead of~\eqref{eq:poly-interp-bound-1}.
However, since the weighting by $\varphi$ in $S_n^\varphi$ yields an interpolant that, at least intuitively, more resembles the elements of $H(K)$ than a polynomial interpolant, it is to be expected that~\eqref{eq:poly-interp-bound-1} is tighter than~\eqref{eq:poly-interp-bound-2}.
In \Cref{sec:applications} we demonstrate that this is indeed the case if $K$ is the Gaussian kernel.
Weighted polynomial interpolation provides a linear algorithm that attains the upper bounds in \Cref{thm:main-3}.

\begin{proof}[Proofs of upper bounds in \Cref{thm:main-3}.]
Let us first consider $p = \infty$.
Let $S_n^\varphi f$ be the weighted polynomial interpolant in~\eqref{eq:polynomial-interpolant-weighted}.
  Proposition~\ref{prop:poly-interp-bound} and Lemma~\ref{lemma:CK-2n-bound} yield
  \begin{equation*}
    \begin{split}
    \lvert f(x) - (S_n^\varphi f)(x) \rvert &\leq \varphi_{\max} \norm[0]{f}_K \frac{(C_{R}^{2n})^{1/2}}{n!} \bigg\lvert \prod_{k=1}^n (x - x_k) \bigg\lvert  \\
    &\leq \varphi_{\max} \norm[0]{f}_K \sqrt{c_L(\lambda)} \, n^{-1/8} e^{\sqrt{\smash[b]{\lambda n}}} \alpha_n^{-1/2} \bigg\lvert \prod_{k=1}^n (x - x_k) \bigg\lvert  \\
    \end{split}
  \end{equation*}
  when $n \geq m_L(\lambda)$.
  The upper bound in~\eqref{eq:main-33} is immediate, while that in~\eqref{eq:main-31} follows from the crude estimate $\lvert \prod_{k=1}^n (x - x_k) \rvert \leq 2^n$.
  To obtain a tighter bound for the $n$th minimal error we can select $x_1, \ldots, x_n$ as the roots of the $n$th Chebyshev polynomial $T_n$, for which we refer to \Cref{sec:lower-Linf}. Because the leading coefficient of $T_n$ is $2^{n-1}$, we have
  \begin{equation} \label{eq:chebyshev-sup-unif-proof}
    \sup_{x \in [-1,1]} \, \abs[3]{ \prod_{k=1}^n (x - x_{k}) } = 2^{-n+1} \norm[0]{T_n}_\infty,
  \end{equation}
  so that the upper bound in~\eqref{eq:main-32} follows from $\norm[0]{T_n}_\infty = 1$.

  The upper bounds for $p = 2$ are proved identically. We obtain
  \begin{equation*}
    \lVert f - S_n^\varphi f \rVert_2 \leq \varphi_{\max} \norm[0]{f}_K \sqrt{c_L(\lambda)} \, n^{-1/8} e^{\sqrt{\smash[b]{\lambda n}}} \alpha_n^{-1/2} \bigg( \int_{-1}^1 \prod_{k=1}^n (x - x_k)^2 \dif x \bigg)^{1/2} .
  \end{equation*}
  Estimating the integral as
  \begin{equation*}
    \begin{split}
      \int_{-1}^1 \prod_{k=1}^n (x - x_k)^2 \dif x &= \int_{-1}^0 \prod_{k=1}^n (x - x_k)^2 \dif x + \int_0^1 \prod_{k=1}^n (x - x_k)^2 \dif x \\
      &\leq \int_{-1}^0 (x - 1)^{2n} \dif x + \int_0^1 (x + 1)^{2n} \dif x \\
      &= \frac{2(2^{2n+1} - 1)}{2n + 1} \\
      &\leq n^{-1} 2^{2n + 1}
      \end{split}
  \end{equation*}
gives the upper bound in~\eqref{eq:main-31}. 
The upper bound in~\eqref{eq:main-32} is obtained by again selecting $x_1, \ldots, x_n$ as the roots of the $n$th Chebyshev polynomial $T_n$.
This improves the above integral estimate to
  \begin{equation} \label{eq:Chebyshec-L2-norm}
    \bigg( \int_{-1}^1 \prod_{k=1}^n (x - x_k)^2 \dif x \bigg)^{1/2} = \bigg( 2^{-2n+2} \int_{-1}^1 T_n(x) \dif x \bigg)^{1/2} \leq 2^{-n + 3/2}, 
  \end{equation}
from which the upper bound of~\eqref{eq:main-32} follows.\footnote{By extremal properties of monic orthogonal polynomials, the constant $2^{3/2}$ in~\eqref{eq:Chebyshec-L2-norm} could be optimised by selecting $x_1, \ldots, x_n$ as the roots of the $n$th Legendre polynomial.}
\end{proof}

\subsection{Application to specific kernels} \label{sec:applications}

In this section we apply~\eqref{eq:main-32} with $p = \infty$ to the four weighted power series kernels mentioned in the beginning of the section.
Derivations for $p = 2$ and other worst-case errors are entirely analogous and thus omitted.
It is notable that for each kernel the rate of decay of the $n$th minimal error is controlled by the shape parameter $\varepsilon$ as in~\eqref{eq:scale-dependency}.

The Gaussian kernel
\begin{equation} \label{eq:app-gaussian}
  K(x, y) = \exp\bigg(\! -\frac{1}{2} \varepsilon^2 (x - y)^2 \bigg)
\end{equation}
is obtained from~\eqref{eq:kernel} by selecting
\begin{equation*}
  \alpha_k = \varepsilon^{-2k} k! \quad \text{ and } \quad \varphi(x) = \exp\bigg(\!-\frac{1}{2} \varepsilon^2 x^2 \bigg).
\end{equation*}

\begin{corollary} \label{cor:gaussian}
  Let $K$ be the Gaussian kernel in~\eqref{eq:app-gaussian} and consider the setting of \Cref{thm:main-3}.
  Then
  \begin{equation*}
    c_1 \bigg( \frac{\varepsilon}{2} \bigg)^n (n!)^{-1/2} \leq e_\infty^{\min}(n) \leq  c_2 n^{-1/8} e^{\varepsilon \sqrt{\smash[b]{n}}} \bigg( \frac{\varepsilon}{2} \bigg)^n (n!)^{-1/2}
  \end{equation*}
  where the lower bound holds when $n \geq 1$ and the upper bound when $n \geq m_L(\varepsilon^2)$ and
  \begin{equation*}
    c_1 = \sqrt{ \frac{2}{1 + \varepsilon^2}} \, e^{-(\frac{1}{32} \varepsilon^4 + \frac{1}{2} \varepsilon^2)} \quad \text{ and } \quad c_2 = 2 \sqrt{c_L(\varepsilon^2)}.
  \end{equation*}
\end{corollary}

The upper bound of \Cref{cor:gaussian} is proved using weighted polynomial interpolation and \Cref{prop:poly-interp-bound}.
For the Gaussian kernel it is however easy to use polynomial interpolation and the estimate~\eqref{eq:poly-interp-bound-2}.
Straightforward differentation of the Taylor series of the Gaussian function yields
\begin{equation*}
    C_K^{2n} = \sup_{x \in [-1, 1]} \frac{\partial^{2n}}{\partial v^n \partial w^n} K(v, w) \biggl|_{\substack{v = x \\ w = x}} = (-1)^n \frac{\dif^{\,2n}}{\dif z^{2n}} e^{-\frac{1}{2} \varepsilon^2 z^2} \biggl|_{z = 0} = \varepsilon^{2n} \frac{(2n)!}{2^n n!}.
\end{equation*}
Plugging this in~\eqref{eq:poly-interp-bound-2} and using Stirling's formula produces
\begin{equation*}
  \norm[0]{f - S_n f}_\infty \leq c \norm[0]{f}_K n^{-1/2} (\sqrt{2 e} \, \varepsilon)^{n} n^{-n/2} \sup_{x \in [-1, 1]} \abs[3]{ \prod_{k=1}^n (x - x_k) }
\end{equation*}
for a positive constant $c$ if $f \in H(K)$.
In contrast, in the proof at the end of \Cref{sec:upper-bounds} we saw that
\begin{equation*}
  \begin{split}
  \norm[0]{ f - S_n^\varphi f }_\infty &\leq \varphi_{\max} \norm[0]{f}_K \sqrt{c_L(\varepsilon^2)} \, n^{-1/8} e^{\varepsilon\sqrt{n}} \alpha_n^{-1/2} \sup_{x \in [-1, 1]} \, \abs[3]{ \prod_{k=1}^n (x - x_k) } \\
  &\leq \tilde{c} \norm[0]{f}_K  n^{-1/8-1/4} e^{\varepsilon\sqrt{\smash[b]{n}}} (\sqrt{e} \, \varepsilon)^{n} n^{-n/2} \sup_{x \in [-1, 1]} \, \abs[3]{ \prod_{k=1}^n (x - x_k) }
  \end{split}
\end{equation*}
for a positive constant $\tilde{c}$, where the second inequality uses Stirling's formula~\eqref{eq:stirling}.
These estimates show that by using weighted polynomial interpolation one obtains an improvement of order $2^{-n/2}$ over polynomial interpolation.

The exponential kernel
\begin{equation} \label{eq:app-exp}
  K(x, y) = e^{\varepsilon x y}
\end{equation}
is obtained from~\eqref{eq:kernel} by selecting
\begin{equation*}
  \alpha_k = \varepsilon^{-k} k! \quad \text{ and } \quad \varphi(x) \equiv 1.
\end{equation*}

\begin{corollary}
  Let $K$ be the exponential kernel in~\eqref{eq:app-exp} and consider the setting of \Cref{thm:main-3}.
  Then
  \begin{equation*}
    c_1 \bigg( \frac{\varepsilon^{1/2}}{2} \bigg)^{n} (n!)^{-1/2} \leq e_\infty^{\min}(n) \leq  c_2 n^{-1/8} e^{\sqrt{\smash[b]{\varepsilon n}}} \bigg( \frac{\varepsilon^{1/2}}{2} \bigg)^{n} (n!)^{-1/2}
  \end{equation*}
  where the lower bound holds when $n \geq 1$ and the upper bound when $n \geq m_L(\varepsilon)$ and
  \begin{equation*}
     c_1 = \sqrt{ \frac{2}{1 + \varepsilon}} \, e^{-\frac{1}{32} \varepsilon^2} \quad \text{ and } \quad c_2 = 2 \sqrt{c_L(\varepsilon)}.
  \end{equation*}
\end{corollary}

Let $\tau > 0$.
The analytic Hermite kernel
\begin{equation} \label{eq:app-hermite}
  K(x, y) = \exp\bigg( \! - \frac{1}{2} \varepsilon^2(x^2 + y^2)  + \tau^2 \varepsilon^2 xy \bigg)
\end{equation}
is obtained from~\eqref{eq:kernel} by selecting
\begin{equation*}
  \alpha_k = (\tau \varepsilon)^{-2k} k! \quad \text{ and } \quad \varphi(x) = \exp\bigg(\!-\frac{1}{2} \varepsilon^2 x^2 \bigg).
\end{equation*}

\begin{corollary}
  Let $K$ be the analytic Hermite kernel in~\eqref{eq:app-hermite} and consider the setting of \Cref{thm:main-3}.
  Then
  \begin{equation*}
    c_1 \bigg( \frac{\tau \varepsilon}{2} \bigg)^n (n!)^{-1/2} \leq e_\infty^{\min}(n) \leq  c_2 n^{-1/8} e^{\tau \varepsilon \sqrt{\smash[b]{n}}} \bigg( \frac{\tau \varepsilon}{2} \bigg)^n (n!)^{-1/2}
  \end{equation*}
  where the lower bound holds when $n \geq 1$ and the upper bound when $n \geq m_L(\tau^2 \varepsilon^2)$ and
  \begin{equation*}
     c_1 = \sqrt{ \frac{2}{1 + \tau^2 \varepsilon^2}} \, e^{-(\frac{1}{32} \tau^4 \varepsilon^4 + \frac{1}{2} \varepsilon^2 ) } \quad \text{ and } \quad c_2 = 2 \sqrt{c_L(\tau^2 \varepsilon^2)}.
  \end{equation*}
\end{corollary}

The Bessel kernel
\begin{equation} \label{eq:app-bessel}
  K(x, y) =  \mathcal{I}_0( 2 \varepsilon \sqrt{\smash[b]{x y}} \, )
\end{equation}
is obtained from~\eqref{eq:kernel} by selecting
\begin{equation*}
  \alpha_k = \varepsilon^{-k} k!^2 \quad \text{ and } \quad \varphi \equiv 1.
\end{equation*}

\begin{corollary}
  Let $K$ be the Bessel kernel in~\eqref{eq:app-bessel} and consider the setting of \Cref{thm:main-3}.
  Then
  \begin{equation*}
    c_1 \bigg( \frac{\varepsilon^{1/2}}{2} \bigg)^{n} (n!)^{-1} \leq e_\infty^{\min}(n) \leq  c_2 n^{-1/8} e^{\sqrt{\smash[b]{\varepsilon n}}} \bigg( \frac{\varepsilon^{1/2}}{2} \bigg)^{n} (n!)^{-1}
  \end{equation*}
  where the lower bound holds when $n \geq 1$ and the upper bound when $n \geq m_L(\varepsilon)$ and
  \begin{equation*}
     c_1 = \sqrt{ \frac{2}{1 + \varepsilon}} \, e^{-\frac{1}{32} \varepsilon^2} \quad \text{ and } \quad c_2 = 2 \sqrt{c_L(\varepsilon)}.
  \end{equation*}
\end{corollary}

\subsection{Derivative information} \label{sec:derivative-information}

It is interesting to compare~\eqref{eq:main-31} to bounds for the worst-case error when information consists of derivative evaluations at a single point $a \in (-1, 1)$.
Here we consider the minimal worst-case error
\begin{equation*}
        e_\infty^{\min}(\mathrm{D}_a, n) = \inf_{\psi_1, \ldots, \psi_n \in L^\infty} \sup_{ 0 \neq f \in H(K) } \frac{ \lVert f - \sum_{k=0}^{n-1} f^{(k)}(a) \psi_k \rVert_\infty }{\lVert f \rVert_{K}} .
\end{equation*}
For brevity we omit the case $p = 2$.

\begin{theorem} \label{thm:derivative-unif}
  Suppose that assumptions~\eqref{eq:phi-assumption} and~\eqref{eq:alpha-assumption} hold.
  Let $c_L$ and $m_L$ be the functions defined in~\eqref{eq:cL-constant} and~\eqref{eq:mL-constant} and $a \in (-1, 1)$.
  Then
  \begin{equation*}
    \varphi_{\min} \bigg( \frac{1 + \abs[0]{a}}{(1 + \lambda a^2)^{1/2}}\bigg)^n \alpha_n^{-1/2} \leq e_\infty^{\min}(\mathrm{D}_a, n) \leq \sqrt{c_L(\lambda)} \, \varphi_{\max} \, n^{-1/8} e^{\sqrt{\smash[b]{\lambda n}}} (1 + \abs[0]{a})^n \alpha_n^{-1/2} ,
  \end{equation*}
 where the lower bound holds when $n \geq 1$ and the upper bound when $n \geq m_L(\lambda)$.
\end{theorem}
\begin{proof}
   We first prove the upper bound.
Consider the weighted Taylor approximation $T_a^n f$ given by
\begin{equation*}
  (T_a^n f)(x) = \varphi(x) \sum_{k=0}^{n-1} \frac{g^{(k)}(a)}{k!} (x - a)^k,
\end{equation*}
where $g = \varphi^{-1} f$.
Taylor's theorem yields
  \begin{equation*}
    \norm[0]{f - T_a^n f}_\infty \leq \varphi_{\max} \frac{\norm[0]{g^{(n)}}_\infty}{n!} \sup_{x \in [-1, 1]} \abs[0]{x - a}^n = \varphi_{\max} \frac{\norm[0]{g^{(n)}}_\infty}{n!} (1 + \abs[0]{a})^n.
  \end{equation*}
  The upper bound follows from \Cref{lemma:rkhs-norm-derivative,lemma:CK-2n-bound} and $\norm[0]{f}_K = \norm[0]{g}_{R}$, with $R$ the unweighted kernel in~\eqref{eq:K0-kernel}.

  To prove the lower bound, let $g_n(x) = (x - a)^n$ and $f_n = \varphi g_n \in H(K)$.
  From the Hilbert space characterisation~\eqref{eq:rkhs-norm} and the binomial theorem we compute
  \begin{align*}
    \norm[0]{f_n}_K^2 = \sum_{k=0}^n \alpha_{n-k} \binom{n}{k}^2 a^{2k} = \alpha_n \sum_{k=0}^n \frac{\alpha_{n-k}}{\alpha_n} \binom{n}{k}^2 a^{2k} &\leq \alpha_n \sum_{k=0}^n \lambda^k \frac{(n-k)!}{n!} \binom{n}{k}^2 a^{2k} \\
    &= \alpha_n \sum_{k=0}^n \binom{n}{k} (\lambda a^2)^k \frac{1}{k!}  \\
    &\leq \alpha_n \sum_{k=0}^n \binom{n}{k} (\lambda a^2)^k  \\
    &= \alpha_n (1 + \lambda a^2)^n.
  \end{align*}
  Clearly,
  \begin{equation} \label{eq:Adan-sup}
    \norm[0]{f_n}_\infty = \norm[0]{ \varphi g_n}_\infty \geq \varphi_{\min} \max\{ g_n(-1), g_n(1) \} = \varphi_{\min} (1 + \abs[0]{a})^n.
  \end{equation}
  The derivatives up to order $n-1$ of $g_n$ vanish at $x = a$.
  This implies that also the derivatives of $f_n$ vanish at $x = a$.
  Thus $\sum_{k=0}^{n-1} f^{(k)}(a) \psi_k \equiv 0$.
  Consequently,
  \begin{equation*}
    e_\infty^{\min}(\mathrm{D}_a, n) = \inf_{\psi_1, \ldots, \psi_n \in L^\infty(\Omega)} \sup_{ 0 \neq f \in H(K) } \frac{ \lVert f - \sum_{k=0}^{n-1} f^{(k)}(a) \psi_k \rVert_\infty }{\lVert f \rVert_{K}} \geq \frac{\norm[0]{f_n}_\infty}{\norm[0]{f_n}_K}.
  \end{equation*}
  Using $\norm[0]{f_n}_K^2 \leq \alpha_n (1 + \lambda a^2)^n$ and~\eqref{eq:Adan-sup} yields the claimed lower bound.
\end{proof}

Note that, up to a constant, the upper bound from \Cref{thm:derivative-unif} coincides with the upper bound on $\sup_{x \in [-1, 1]} e^{\min}(x \mid x_1, \ldots, x_n)$ that is obtained by formally setting $x_1 = \cdots = x_n = a$ in~\eqref{eq:main-33}.

\subsection{Some remarks on higher dimensions} \label{sec:higher-dimensions}

Let us conclude with a few short remarks about how the results in this section could and could not be generalised to higher dimensions.
The natural generalisation of the kernel~\eqref{eq:kernel} to dimension $d \in \N$ is
\begin{equation} \label{eq:kernel-multivariate}
  K_d(\b{x}, \b{y}) = \varphi(\b{x}) \varphi(\b{y}) \sum_{ \b{k} \in \N_0^d} \alpha_{\b{k}}^{-1} \b{x}^{\b{k}} \b{y}^{\b{k}} \quad \text{ for } \quad \b{x}, \b{y} \in [-1, 1]^d,
\end{equation}
with summation over $d$-dimensional non-negative multi-indices.
This kernel is well-defined and strictly positive-definite if $\varphi \colon [-1, 1]^d \to \R$ is non-vanishing and $\sum_{\b{k} \in \N_0^d} \alpha_{\b{k}}^{-1} < \infty$.
The definitions of worst-case errors and the Hilbert space characterisations in \Cref{sec:intro} generalise naturally.
Due to non-uniqueness of polynomial interpolation in dimensions higher than one it is no surprise that the worst-case error no longer tends to zero for an arbitrary sequence of points.

\begin{proposition}
  Suppose that $d \geq 2$ and $p \in [0, \infty]$.
  Let $K_d$ be any kernel of the form~\eqref{eq:kernel-multivariate}.
  Then there exists a positive constant $c$ such that for every $n \geq 1$ there are distinct points $\b{x}_1, \ldots, \b{x}_n \in [-1, 1]^d$ such that
  \begin{equation*}
    e_p^{\min}(\b{x}_1, \ldots, \b{x}_n) \geq c .
  \end{equation*}
\end{proposition}
\begin{proof}
  Let $\b{x}_1, \ldots, \b{x}_n$ be any pairwise distinct points on the unit circle.
  The polynomial $P(\b{x}) = x_1^2 + \cdots + x_d^2 - 1$, where $\b{x} = (x_1, \ldots, x_d)$,  vanishes on the unit circle and the function defined as $f(\b{x}) = \varphi(\b{x}) P(\b{x})$ is a non-zero element of $H(K_d)$.
  Therefore $\sum_{k=1}^n f(\b{x}_k) \psi_k \equiv 0$.
  It follows that $e_p^{\min}(\b{x}_1, \ldots, \b{x}_n) \geq \norm[0]{f}_p / \norm[0]{f}_{K_d} > 0$, where the lower bound does not depend on $n$.
\end{proof}

It would be straightforward to use tensor grids to prove limited generalisations to higher dimensions of \Cref{thm:main-3}.
We also note that Bernstein~\cite{Bernstein1948} has proved a certain multivariate extension of the Markov inequality used in \Cref{sec:lower-Linf}. See~\cite{Ganzburg2021} for a somewhat more accessible source.

%% file: sec4.tex
Let $\Phi \colon \R \to \R$ be a continuous and integrable positive-semidefinite function and consider the positive-semidefinite kernel $K$ given by $K(x, y) = \Phi(x - y)$ for all $x, y \in \R$.
Such kernels are variously called stationary or translation-invariant.
Bochner's theorem~\citep[Thm.\@~6.6]{Wendland2005} ensures that the Fourier transform
\begin{equation*}
        \widehat{\Phi} (\omega) = \frac{1}{\sqrt{2\pi}} \int_\R \Phi(x) e^{\ri x \omega} \dif x 
\end{equation*}
of $\Phi$ is real and non-negative.
The Hilbert space $H(K)$ consists of square-integrable functions whose Fourier transforms decay sufficiently fast~\citep[Thm.\@~10.12]{Wendland2005}:
\begin{equation} \label{eq:rkhs-stationary}
        H(K) = \Set[\bigg]{ f \in L^2(\R) }{ \lVert f \rVert_K^2 = \frac{1}{\sqrt{2\pi}} \int_\R \frac{\lvert \widehat{f}(\omega) \rvert^2 }{\widehat{\Phi}(\omega)} \dif \omega < \infty } .
\end{equation}
The Gaussian kernel $K(x, y) = \exp(-\tfrac{1}{2} \varepsilon^2 (x - y)^2)$ is one of the most straightforward examples of a stationary kernel, having
\begin{equation} \label{eq:gaussian-stationary}
        \Phi(z) = \exp\bigg(\! -\frac{1}{2} \varepsilon^2 z^2 \bigg) \quad \text{ and } \quad \widehat{\Phi}(\omega) = \frac{1}{\varepsilon} \exp\bigg(\! -\frac{1}{2\varepsilon^2} \omega^2 \bigg) .
\end{equation}
By~\eqref{eq:rkhs-stationary} the Hilbert space of the Gaussian kernel therefore consists of functions with super-exponentially decaying Fourier transforms.
Note that exponential decay of the Fourier transform suffices for analyticity by the Paley--Wiener theorem~\cite[Thm.\@~IX.13]{ReedSimon1975}.

In this section we prove lower bounds on the worst-case errors in~\eqref{eq:wce1}--\eqref{eq:wce-x} under assumptions analogous to~\eqref{eq:phi-assumption} and~\eqref{eq:alpha-assumption}.
Firstly, we assume that there is a positive constant $\Phi_\textup{min}$ such that 
\begin{equation} \label{eq:phi-assumption-stat}
  0 < \Phi_\textup{min} \leq \abs[0]{\Phi(x)}
\end{equation}
for all $x \in [-1, 1]$.
We note that $\Phi$ is uniformly bounded from above on $[-1, 1]$ by continuity.
Secondly, define 
\begin{equation} \label{eq:alpha-stationary}
        \beta_k = \frac{1}{\sqrt{2\pi}} \int_\R \frac{ \widehat{\Phi}^{(k)}(\omega)^2 }{\widehat{\Phi}(\omega)} \dif \omega \quad \text{ for } \quad k \geq 0.
\end{equation}
We assume that $\widehat{\Phi}$ is infinitely differentiable and that there exists a positive constant $\lambda$ such that
\begin{equation} \label{eq:alpha-assumption-stat}
  \beta_n < \infty \quad \text{ and } \quad \frac{\beta_{k}}{\beta_n} \leq \lambda^{n-k} \frac{k!}{n!}
\end{equation}
for all $n \geq 0$ and $0 \leq k \leq n$.
Observe that in \Cref{sec:power-series} the coefficient $\alpha_k$ is the squared Hilbert space norm of the weighted monomial $f(x) = \varphi(x) x^k$ by~\eqref{eq:rkhs-norm}.
Similarly, for $g(x) = \Phi(x) x^k$ we get
\begin{equation} \label{eq:mono-norm-stat}
        \lVert g \rVert_K^2 = \frac{1}{\sqrt{2\pi}} \int_\R \frac{\lvert \widehat{g}(\omega) \rvert^2 }{\widehat{\Phi}(\omega)} \dif \omega = \frac{1}{\sqrt{2\pi}} \int_\R \frac{\lvert \ri^k \widehat{\Phi}^{(k)}(\omega) \rvert^2 }{\widehat{\Phi}(\omega)} \dif \omega = \beta_k .
\end{equation}
The constants $\beta_k$ and assumption~\eqref{eq:alpha-assumption-stat} play the same role as $\alpha_k$ and~\eqref{eq:alpha-assumption} in \Cref{sec:power-series}.

Before proceeding, let us verify~\eqref{eq:alpha-assumption-stat} for the Gaussian kernel. Since $\varphi = \Phi$ for the Gaussian kernel, it follows from~\eqref{eq:mono-norm-stat} that $\beta_k = \alpha_k = \varepsilon^{-2k} k!$, but it is nevertheless useful to use Fourier transforms to check this.
Equation~\eqref{eq:gaussian-stationary} and the Rodrigues' formula for probabilist's Hermite polynomials give
\begin{equation*}
        \widehat{\Phi}^{(k)}(\omega) = \frac{(-1)^k}{\varepsilon^{k+1}} H_k\bigg(\frac{\omega}{\varepsilon}\bigg) \exp\bigg(\! -\frac{1}{2\varepsilon^2} \omega^2 \bigg) = \frac{(-1)^k}{\varepsilon^{k}} H_k\bigg(\frac{\omega}{\varepsilon}\bigg) \widehat{\Phi}(\omega) ,
\end{equation*}
where $H_k$ is the $k$th probabilist's Hermite polynomial.
From the weighted $L^2$ norm of Hermite polynomials we obtain
\begin{equation*}
        \begin{split}
        \beta_k = \frac{1}{\sqrt{2\pi}} \int_\R \frac{ \widehat{\Phi}^{(k)}(\omega)^2 }{\widehat{\Phi}(\omega)} \dif \omega &= \varepsilon^{-2k} \frac{1}{\sqrt{2\pi}} \int_\R H_k\bigg(\frac{\omega}{\varepsilon}\bigg)^2 \, \frac{1}{\varepsilon} \exp\bigg(\! -\frac{1}{2\varepsilon^2} \omega^2 \bigg) \dif \omega \\
        &= \varepsilon^{-2k} \frac{1}{\sqrt{2\pi}} \int_\R H_k(\omega)^2 \exp\bigg(\! -\frac{1}{2} \omega^2 \bigg) \dif \omega \\
        &= \varepsilon^{-2k} k! ,
        \end{split}
\end{equation*}
which satisfies~\eqref{eq:alpha-assumption-stat} with $\lambda = \varepsilon^2$.
We prove the following theorem in \Cref{sec:proof-4}.

\begin{theorem} \label{thm:main-4}
  Let $K$ be a stationary positive-semidefinite kernel and $p \in \{2, \infty\}$.
  Suppose that $\Phi$ and $(\beta_k)_{k=0}^\infty$ satisfy~\eqref{eq:phi-assumption-stat} and~\eqref{eq:alpha-assumption-stat}.
  Then
  \begin{align}
        \tilde{c}_{1,p} \, (n+1)^{-1/2} \, 2^{-n} \beta_n^{-1/2} &\leq e_p^{\min}(x_1, \ldots, x_n), \label{eq:main-41} \\
        \tilde{c}_{1,p} \, (n+1)^{-1/2} \, 2^{-n} \beta_n^{-1/2} &\leq e_p^{\min}(n), \label{eq:main-42} \\
        \tilde{c}_{1,\infty} \, (n+1)^{-1/2} \, 2^{-2n} \beta_n^{-1/2} &\leq \frac{e^{\min}(x \mid x_1, \ldots, x_n)}{\lvert \prod_{k=1}^n (x - x_k) \rvert} \label{eq:main-43}
  \end{align}
  for every $n \geq 1$, where
  \begin{equation} \label{eq:constants-stat}
    \tilde{c}_{1,\infty} = \sqrt{ \frac{2}{1 + \lambda}} \, e^{-\frac{1}{32}\lambda^2} \Phi_\textup{min} \quad \text{ and } \quad \tilde{c}_{1, 2} = \frac{2\sqrt{2\pi}}{3e^3 \sqrt{\smash[b]{3(4 + \lambda)}}} \, e^{-\frac{1}{32}\lambda^2} \Phi_\textup{min}.
  \end{equation}
\end{theorem}

The only difference to the lower bounds of \Cref{thm:main-3} is the additional factor $(n+1)^{-1/2}$ and the substitution of $\Phi_\textup{min}$ for $\varphi_\textup{min}$ in the constant.
We omit upper bounds as we have not found a convenient way to express them in terms of $\beta_n$.
Note that polynomial interpolation~\eqref{eq:polynomial-interpolation-error} and \Cref{lemma:rkhs-norm-derivative}, together with stationarity, yield the bound 
\begin{equation*}
        \begin{split}
        \lvert f(x) - (S_n f)(x) \rvert &\leq \norm[0]{f}_K \frac{(C_{K}^{2n})^{1/2}}{n!} \, \abs[3]{ \prod_{k=1}^n (x - x_k) } \\
        &= \norm[0]{f}_K \abs[3]{ \prod_{k=1}^n (x - x_k) } \frac{1}{n!} \sqrt{\smash[b]{\Phi^{(2n)}(0)}} \\
        &= \norm[0]{f}_K \abs[3]{ \prod_{k=1}^n (x - x_k) } \frac{1}{n!} \bigg( \int_\R \omega^{2n} \, \widehat{\Phi}(\omega) \dif \omega \bigg)^{1/2} ,
        \end{split}
\end{equation*}
which is in terms of the moments of $\widehat{\Phi}$ rather than of $\beta_n$.
The upper bound in~\cite{Yarotsky2013} (to be discussed in \Cref{sec:yarotsky}) is derived from the above estimate.

\subsection{Proof of \Cref{thm:main-4}} \label{sec:proof-4}

To prove \Cref{thm:main-4} we use the following lemma, a cousin of \Cref{lemma:unif-norm-bound,lemma:l2-norm-bound}.

\begin{lemma} \label{lemma:rkhs-norm-bound-stat}
Let $P_n(x) = \sum_{k=0}^n a_k x^k$ for $a_k \in \R$ be a polynomial of degree $n$.
If assumption~\eqref{eq:alpha-assumption-stat} holds, then the function $f_n(x) = \Phi(x) P_n(x)$ is an element of $H(K)$ and its norm satisfies
\begin{equation*}
        \norm[0]{f_n}_K^2 \leq (1 + \lambda) \, \tau_2 \, (n+1) \, 2^{2n} \alpha_n \norm[0]{P_n}_\infty^2 \quad \text{ and } \quad \norm[0]{f_n}_K^2 \leq \ell_2 \, (n + 1) \,  2^{2n} \alpha_n \norm[0]{P_n}_{2}^2
\end{equation*}
for every $n \geq 1$, where the constants $\tau_2$ and  $\ell_2$ are given in~\eqref{eq:chebyshev-norm-lemma-constants} and~\eqref{eq:c-ell-L2}.
\end{lemma}
\begin{proof}
By the properties of the Fourier transform,
\begin{equation*}
        \widehat{f}_n(\omega) = \sum_{k=0}^n a_k \ri^k \widehat{\Phi}^{(k)}(\omega).
\end{equation*}
From~\eqref{eq:rkhs-stationary}, Jensen's inequality and~\eqref{eq:alpha-stationary} we get
\begin{equation*}
        \begin{split}
        \lVert f_n \rVert_K^2 = \frac{1}{\sqrt{2\pi}} \int_\R \frac{\lvert \widehat{f}_n(\omega) \rvert^2 }{\widehat{\Phi}(\omega)} \dif \omega &= \frac{1}{\sqrt{2\pi}} \int_\R \frac{1}{\widehat{\Phi}(\omega)} \bigg\lvert \sum_{k=0}^n a_k \ri^k \widehat{\Phi}^{(k)}(\omega) \bigg\rvert^2 \dif \omega \\
        &\leq \frac{n+1}{\sqrt{2\pi}} \int_\R \frac{1}{\widehat{\Phi}(\omega)} \sum_{k=0}^n a_k^2 \, \widehat{\Phi}^{(k)}(\omega)^2 \dif \omega \\
        &= (n+1) \sum_{k=0}^n \beta_k a_k^2 .
        \end{split}
\end{equation*}
\Cref{lemma:unif-norm-bound,lemma:l2-norm-bound} provide bounds on $\sum_{k=0}^n \alpha_k a_k^2$ if $\alpha_k$ satisfy~\eqref{eq:alpha-assumption}.
Because assumption~\eqref{eq:alpha-assumption-stat} on $\beta_k$ is identical to~\eqref{eq:alpha-assumption}, the claimed bounds follow from these lemmas.
\end{proof}

\begin{proof}[Proofs of lower bounds in \Cref{thm:main-4}] 
Select $f_n(x) = \Phi(x) (x - x_1) \cdots (x - x_n)$ and proceed as in \Cref{sec:lower-Linf,sec:lower-L2}. 
Use \Cref{lemma:rkhs-norm-bound-stat} in the place of \Cref{lemma:unif-norm-bound,lemma:l2-norm-bound} to estimate the norm of $f_n$, which yields lower bounds that differ from those in \Cref{thm:main-3} by the factor $(n+1)^{-1/2}$.
\end{proof}

\subsection{Bounds by Yarotsky} \label{sec:yarotsky}

Yarotsky~\cite{Yarotsky2013} has derived pointwise upper and lower bounds on the posterior variance in Gaussian process interpolation for a certain class of analytic stationary covariance kernels.
It is well known that the Gaussian process posterior variance equals the squared worst-case error in the Hilbert space of the covariance kernel, a fact that one can verify by noting that~\cite[e.g.,][Sec.\@~10.2]{NovakWozniakowski2010}
\begin{equation*} 
        \begin{split}
        e^{\min}(x \mid x_1, \ldots, x_n)^2 &= \inf_{u_1, \ldots, u_n \in \R} \sup_{ 0 \neq f \in H(K) } \frac{ \lvert f(x) - \sum_{k=1}^n f(x_k) u_k \rvert^2 }{\lVert f \rVert_{K}^2} \\
        &= \inf_{u_1, \ldots, u_n \in \R} \bigg\lVert K(\cdot, x) - \sum_{k=1}^n u_k K(\cdot, x_k) \bigg\rVert_K^2 \\
        &= \inf_{u_1, \ldots, u_n \in \R} \bigg( K(x, x) - 2\sum_{k=1}^n u_k K(x, x_k) + \sum_{k,l=1}^n u_k u_l K(x_k, x_l) \bigg) \\
        &= K(x, x) - \mathsf{k}_n(x)^\mathsf{T} \mathsf{K}_n^{-1} \mathsf{k}_n(x) ,
        \end{split}                   
\end{equation*}
where $\mathsf{k}_n(x)$ is an $n$-vector with elements $K(x, x_i)$ and $\mathsf{K}_n$ an $n \times n$ matrix with elements $K(x_i, x_j)$.
The last expression, which requires assuming that the kernel is strictly positive-definite, equals the posterior variance in Equation~(6) of~\cite{Yarotsky2013}.

\begin{theorem}[Special case of Theorem~2 in \cite{Yarotsky2013}] \label{thm:yarotsky}
Let $K$ be a strictly positive-definite stationary kernel such that
\begin{equation*}
        \widehat{\Phi}(\omega) = \exp(-a\lvert \omega \rvert^b) \quad \text{ for } \quad a > 0 \: \text{ and } \: b > 1.
\end{equation*}
If $n$ is sufficiently large and $x_1, \ldots, x_n \in [-1, 1]$ are distinct, then
\begin{equation*}
        \begin{split}
        c \, e^{-n/2} (ab e)^{-n/b} &\bigg(\frac{(2n+1)^{1/b}}{n} \bigg)^{n+1/2} \\
        &\leq \frac{e^{\min}(x \mid x_1, \ldots, x_n)}{\lvert \prod_{k=1}^n (x - x_k) \rvert} \leq c \, e^{n} (ab e)^{-n/b} \bigg(\frac{(2n+1)^{1/b}}{n} \bigg)^{n+1/2} ,
        \end{split}          
\end{equation*}
where $c = (ab e)^{-1/2b}$.
\end{theorem}

Yarotsky uses polynomial interpolation to prove the upper bound.
His proof of the lower bound differs considerably from our proofs based on estimation of polynomial coefficients.
Rather than construct a fooling function, he writes the worst-case error as an integral of the product of an exponential sum and the Fourier transform of $\Phi$ which he then truncates and bounds from below, a feat which involves an explicit expression for the distance between $(1, z, \ldots, z^{n}) \in \C^{n+1}$ and the span of $((1, z_k, \ldots, z_k^n))_{k=1}^n \subset \C^{n+1}$.

The assumption $b > 1$ ensures that $H(K)$ consists of analytic functions.
Setting $b = 2$ and $a = 1/(2\varepsilon^2)$ gives a constant multiple of the Gaussian kernel by~\eqref{eq:gaussian-stationary}.
In this case the lower and upper bounds in \Cref{thm:yarotsky} are asymptotically constant multiples of
\begin{equation*}
        n^{-1/4} \bigg(\frac{\sqrt{2} \, \varepsilon}{e}\bigg)^n n^{-n/2} \quad \text{ and } \quad n^{-1/4} (\sqrt{2e} \, \varepsilon)^n n^{-n/2} ,
\end{equation*}
respectively.
For the Gaussian kernel our corresponding bounds in~\eqref{eq:main-33} are constant multiples of
\begin{equation} \label{eq:yaro-ours}
        \bigg( \frac{\varepsilon}{4} \bigg)^n (n!)^{-1/2} \quad \text{ and } \quad n^{-1/8} e^{\varepsilon \sqrt{n}} \varepsilon^n (n!)^{-1/2} .
\end{equation}
By Stirling's formula, $n! \sim \sqrt{2\pi} n^{n+1/2} e^{-n}$, so that the bounds~\eqref{eq:yaro-ours} are asymptotically constant multiples of
\begin{equation*}
      n^{-1/4} \bigg( \frac{\sqrt{e} \, \varepsilon}{4} \bigg)^n n^{-n/2} \quad \text{ and } \quad n^{-3/8} e^{\varepsilon \sqrt{n}} (\sqrt{e} \, \varepsilon)^n n^{-n/2} .
\end{equation*}
Because \smash{$\sqrt{2}/e \approx 0.52 > \sqrt{e}/4 \approx 0.41$} and \smash{$\sqrt{e} < \sqrt{2e}$}, the lower bound in \Cref{thm:yarotsky} is tighter than that in~\eqref{eq:main-33} but the upper bound is looser.
That our upper bound is tighter is due to our use of weighted polynomial interpolation; Yarotsky uses polynomial interpolation, which is expected to yield worse upper bounds as discussed in \Cref{sec:upper-bounds,sec:applications}.